\documentclass[11pt]{article}
\usepackage{e-jc}
\usepackage{amsthm,amsmath,amssymb}

\usepackage[colorlinks=true,citecolor=black,linkcolor=black,urlcolor=blue]{hyperref}

\usepackage{graphicx}

\numberwithin{equation}{section}
\numberwithin{figure}{section}
\theoremstyle{plain}
\newtheorem{thm}{Theorem}[section]
\newtheorem{lem}[thm]{Lemma}

\newtheorem{conj}[thm]{Conjecture}

\theoremstyle{remark}

\newcommand{\M}{\operatorname{M}}

\title{On the numbers of perfect matchings of trimmed Aztec rectangles}
\author{Tri Lai\footnote{This research was supported in part by the Institute for Mathematics and its Applications with funds provided by the National Science Foundation.}\\
\small Institute for Mathematics and its Applications\\[-0.8ex]
\small University of Minnesota\\[-0.8ex]
\small Minneapolis, MN 55455\\
\small\tt tmlai@ima.umn.edu
}

\date{\small Mathematics Subject Classifications: 05A15,  05C30, 05C70}

\begin{document}
\maketitle

\begin{abstract}
We consider several new families of graphs  obtained from Aztec rectangle and augmented Aztec rectangle graphs by trimming two opposite corners. We prove that the perfect matchings of these new graphs are enumerated by powers of $2$, $3$, $5$, and $11$. The result yields a proof of a conjectured posed by Ciucu. In addition, we reveal a hidden relation between our graphs and the hexagonal dungeons introduced by Blum.

\bigskip\noindent \textbf{Keywords:} perfect matching, tiling, dual graph, Aztec rectangle, graphical condensation, hexagonal dungeon.
\end{abstract}

\section{Introduction and main results}

Consider a $\sqrt{2}m\times\sqrt{2}n$ rectangular contour rotated $45^0$ and translated so that its vertices are centers of some unit squares on the square grid. The $m\times n$ \textit{Aztec rectangle (graph)}\footnote{From now on, the term ``Aztec rectangle" will be used to mean ``Aztec rectangle graph".} $AR_{m,n}$ is the subgraph of the square grid induced by the vertices inside or on the boundary of the rectangular contour. The graph restricted in the bold contour on the left of Figure \ref{TrimARnew} shows the Aztec rectangle $AR_{6,8}$.

\begin{figure}\centering
\includegraphics[width=12cm]{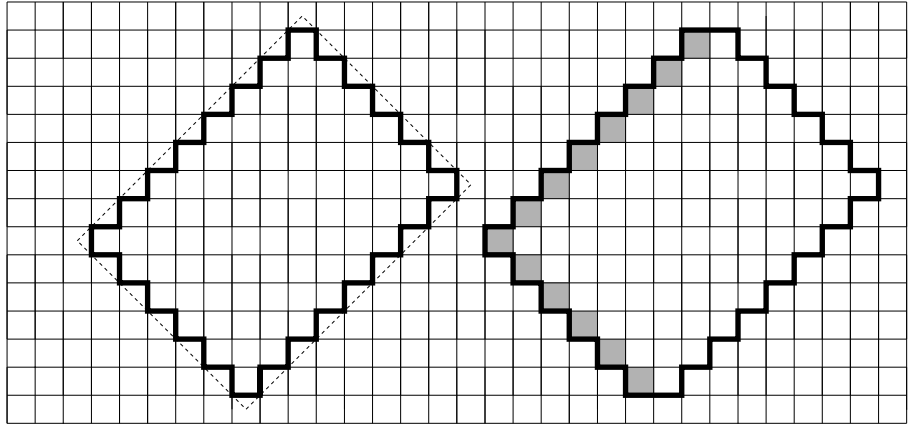}
\caption{The Aztec rectangle $AR_{6,8}$ (left) and the augmented Aztec rectangle $AAR_{6,8}$ (right).}
\label{TrimARnew}
\end{figure}

\begin{figure}\centering
\includegraphics[width=10cm]{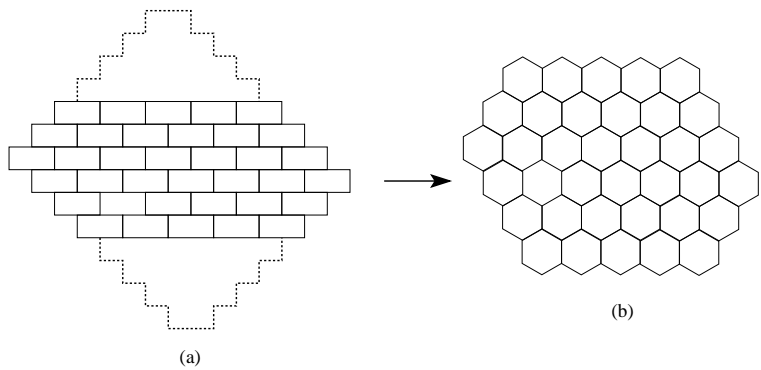}
\caption{Obtaining a honeycomb graph from a trimmed augmented Aztec rectangle on the brick lattice.}
\label{trimhoney}
\end{figure}

The \textit{augmented Aztec rectangle} $AA_{m,n}$ is obtained by ``stretching" the Aztec rectangle $AR_{m,n}$ one unit horizontally, i.e.  adding one square to the left of each row in the Aztec rectangle (see the graph restricted by the black contour on the right of  Figure \ref{TrimARnew}; the added squares are shaded ones). We can still define the Aztec rectangle and augmented Aztec rectangle on sub-grids or weighted versions of the square grid.

A \textit{perfect matching} of a graph $G$ is a collection of edges of $G$ so that each vertex is incident precisely one edge in the collection. A perfect matching is sometimes called 1-factor (in graph theory) or dimmer covering (in statistical mechanics). In this paper we use the notation $\M(G)$ for the number of perfect matchings of a graph $G$. We are interested in how many different perfect matchings in a particular graph.

The Aztec rectangle $AR_{m,n}$ has $2^{n(n+1)/2}$ perfect matchings when $n=m$ (see \cite{Elkies}), and $0$ perfect matching otherwise. The Aztec rectangle $AR_{m,n}$ is called the \textit{Aztec diamond} of order $n$ when $m=n$; and similarly $AA_{n,n}$ is called the \textit{augmented Aztec diamond} of order $n$. Sachs and Zernitz (\cite{SZ}) proved that the augmented Aztec diamond of order $n$ has $D(n,n)$ perfect matchings, where the Delannoy number $D(m,n)$, for $m,n\geq 0$, is the number of lattice paths on $\mathbb{Z}^2$ from the vertex $(0,0)$ to the vertex $(m,n)$ using north, northeast and east steps (see Exercise 6.49 in \cite{Stanley}; strictly speaking the exercise asks for the number of  tilings of a region, however, the tilings are in bijection with the perfect matchings of our graph $AA_{n,n}$).  Dana Randall later gave a simple combinatorial proof for the result. By a similar argument, one can show that the number of perfect matchings of $AA_{m,n}$ is also given by $D(m,n)$, for any $m,n\geq 0$.

 Many other interesting results on perfect matchings of Aztec rectangles and its variations have been proven, focused on graphs whose numbers of perfect matchings are given by simple product formulas (see e.g, \cite{Ciucu}, \cite{Elkies}, \cite{Propp}, \cite{Yang},  \cite{Krat}, \cite{Tri2}).

Here is a simple observation that inspires our main results. Viewing a standard brick lattice as a sub-grid of the square grid, we consider an augmented Aztec rectangle on the standard brick lattice, where the north and the south corners have been trimmed (see Figure \ref{trimhoney}(a)). One readily sees the resulting graph can be deformed into the honeycomb graph whose perfect matchings are enumerated by MacMahon's formula \cite{Mac} (see Figure \ref{trimhoney}(b)).

\begin{figure}\centering
\includegraphics[width=13cm]{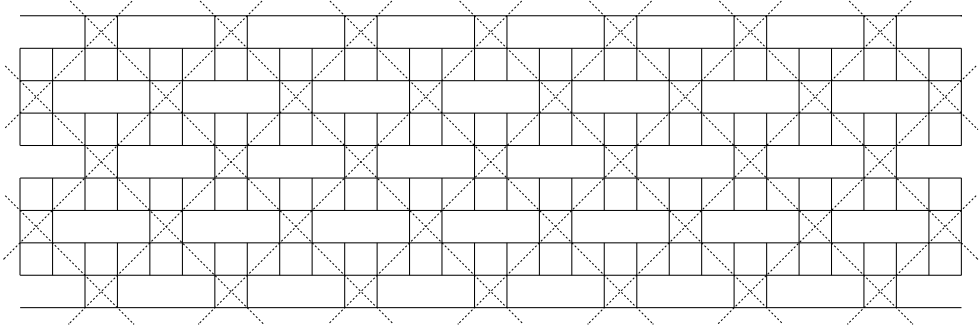}
\caption{The $B$ lattice (solid lines), and its partition into crosses (restricted by dotted diamonds).}
\label{Blattice}
\end{figure}

Next, we consider a new sub-grid $B$ of the square grid pictured in Figure \ref{Blattice}. In particular, the grid $B$ is obtained by gluing copies of a cross pattern, which is restricted in a dotted diamond of side $2\sqrt{2}$.  Let $a$ and $b$ be two non-negative integers. Consider a $(2b+2a-2)\times(2b+4a-2)$ augmented Aztec rectangle on the new grid so that its east-most edge is the east-most edge of a cross. Motivated by the observation in the previous paragraph,  we trim the north and south corners of the graph both from right to left at levels $(2a-1)$ above and $(4a-1)$ below the eastern corner. The only difference here is that we trim by zigzag cuts containing alternatively bumps and holes of size $2$, as opposed to straight lines in the previous paragraph (see Figure \ref{Trimdungeon}). Denote by $TR_{a,b}$ the resulting graph. The number of perfect matchings of $TR_{a,b}$ is given by the theorem stated below.

\begin{figure}\centering
\includegraphics[width=10cm]{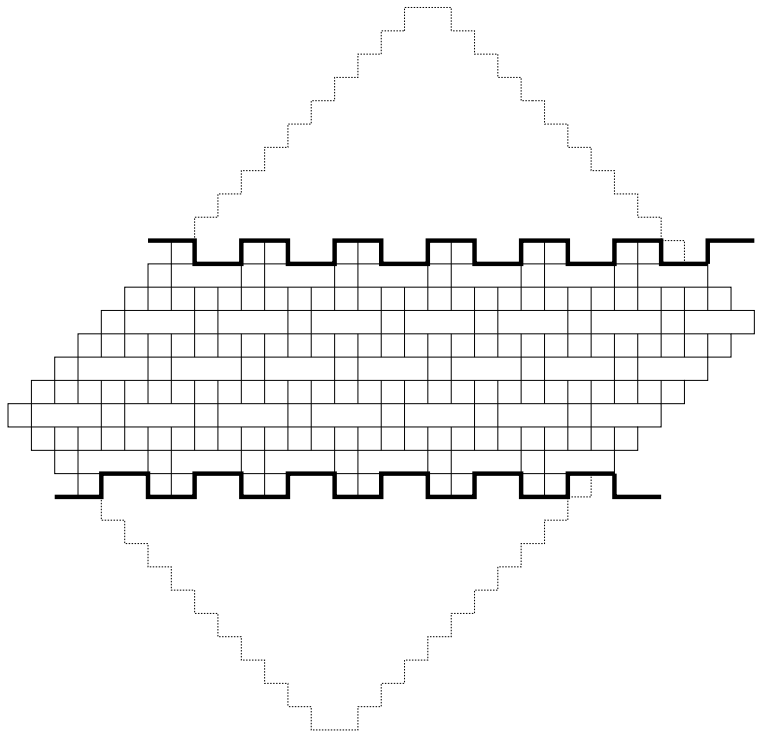}
\caption{The trimmed augmented Aztec rectangle $TR_{2,6}$ on the lattice $B$.}
\label{Trimdungeon}
\end{figure}

\begin{thm}\label{Dungeon}
Assume that $a$ and $b$ are positive integers so that $b\geq 2a$. Then the number of perfect matchings of $TR_{a,b}$ is $10^{8k^2}11^{2k^2}$ if $a=2k$, and $10^{2(2k+1)^2}11^{2k(k+1)}$ if $a=2k+1$.
\end{thm}
We notice that the number of perfect matchings of $TR_{a,b}$ in Theorem \ref{Dungeon} does not depend on $b$.

\begin{figure}\centering
\includegraphics[width=10cm]{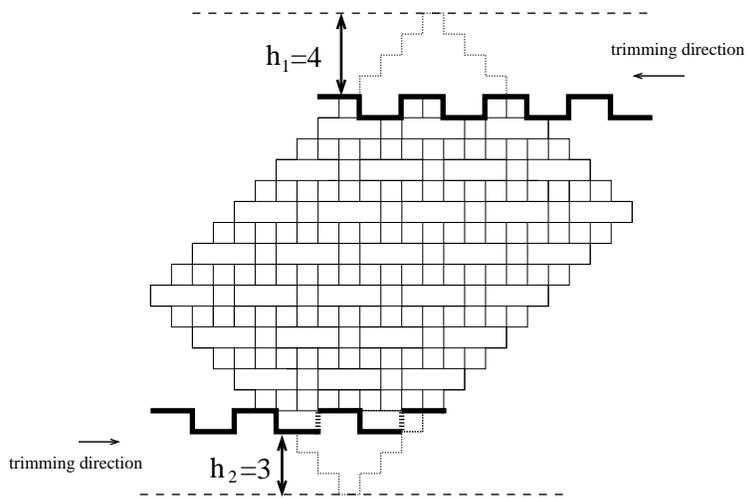}
\caption{The graph $TA_{5,7}^{4,3}$ is obtained by trimming the rectangle $AR_{10,14}$ (on the grid $B$) by two zigzag cuts.}
\label{trimrectangle}
\end{figure}

\begin{figure}\centering
\includegraphics[width=8cm]{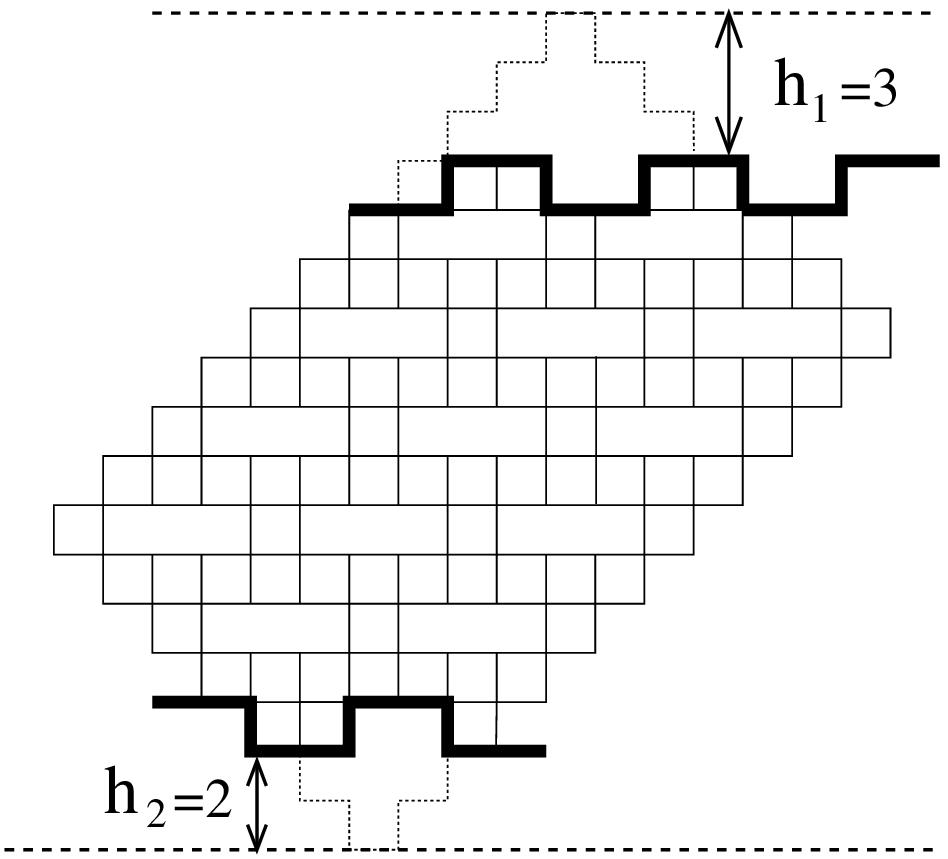}
\caption{The graph $TB_{4,7}^{3,6}$ is obtained by trimming the rectangle $AR_{7,11}$ (on the grid $B$) by trimming two corners.}
\label{trimrectangle5}
\end{figure}

We now consider a $2m\times 2n$ Aztec rectangle ($m\leq n$) on the grid $B$ so that its east-most edge is also the east-most edge of a cross pattern. Similar to the graph $TR_{a,b}$, we also trim the upper and lower corners of the Aztec rectangle by two zigzag cuts from right to left and from left to right, respectively. Assume that $h_1$ is the distance between the top of the Aztec rectangle and the upper cut, and that $h_2$ is the distance between the bottom of the graph and the lower cut. Denote by $TA_{m,n}^{h_1,h_2}$ the resulting trimmed Aztec rectangle. We also have a variant of $TA_{m,n}^{h_1,h_2}$ by applying the same process to the $(2m-1)\times (2n-1)$ Aztec rectangle whose east-most edge is the \textit{west-most} edge of a cross pattern on $B$. Denote by $TB_{m,n}^{h_1,h_2}$ the resulting graph (see Figures \ref{trimrectangle}). Ciucu conjectured that
\begin{conj}
The numbers  of perfect matchings of $TA_{m,n}^{h_1,h_2}$ and $TB_{m,n}^{h_1,h_2}$ have only prime factors less than $13$ in their prime factorizations.
\end{conj}

We prove the Ciucu's conjecture by giving  explicit formulas for the numbers of perfect matchings of $TA_{m,n}^{h_1,h_2}$ and $TB_{m,n}^{h_1,h_2}$ as follows.
Given three integer numbers $a,b,c$, we define five functions $g(a,b,c)$, $q(a,b,c)$, $\alpha(a,b,c)$, $\beta(a,b,c)$, and $t(a,b)$ by setting
\begin{equation}
g(a,b,c):=(b-a)(b-c)+\left\lfloor\frac{(a-c)^2}{3}\right\rfloor,
\end{equation}
\begin{equation}
q(a,b,c):=  \left\lfloor\frac{(a-b+c)^2}{4}\right\rfloor,
\end{equation}
\begin{equation}
\alpha(a,b,c):=
\begin{cases}
2 &\text{if $3b+a-c\equiv 1 \pmod{6}$;}\\
3 &\text{if $3b+a-c\equiv 5 \pmod{6}$;}\\
1 &\text{otherwise,}
\end{cases}
\end{equation}
 \begin{equation}
\beta(a,b,c):=
\begin{cases}
3 &\text{if $3b+a-c\equiv 1 \pmod{6}$;}\\
2 &\text{if $3b+a-c\equiv 5 \pmod{6}$;}\\
1 &\text{otherwise,}
\end{cases}
\end{equation}
and
\begin{equation}
\tau(a,b):=
\begin{cases}
\frac{a+b}{2} &\text{if $h_1$ and $h_2$ are even;}\\
\frac{a}{2} &\text{if $h_1$ is even and $h_2$ is odd;}\\
\frac{b}{2} &\text{if $h_1$ is odd and $h_2$ is even;}\\
0 &\text{otherwise.}
\end{cases}
\end{equation}

\begin{thm}\label{MCconj}
 Assume that $\left\lfloor\frac{h_1+1}{2}\right\rfloor+\left\lfloor\frac{h_2+1}{2}\right\rfloor=2(n-m)$.

 (a) The number of perfect matchings of the trimmed Aztec rectangle $TA_{m,n}^{h_1,h_2}$ equals
\begin{equation}\label{trimeq1}
\alpha(a,b,c)2^{g(a,b,c+1)}3^{\tau(h_1,h_2)}5^{g(a,b,c)}11^{q(a,b,c)},
\end{equation}
where $a=m-\left\lfloor\frac{h_1+1}{2}\right\rfloor+1$,
$b=n-\left\lfloor\frac{h_1+1}{2}\right\rfloor+1$ and $c=\left\lfloor\frac{h_2+1}{2}\right\rfloor$.

(b) The number of perfect matchings of $TB_{m,n}^{h_1,h_2}$ is
\begin{equation}\label{trimeq2}
\beta(a',b',c')2^{g(a',b',c'-1)}3^{\tau(h_1+1,h_2+1)}5^{g(a',b',c')}11^{q(a',b',c')},
\end{equation}
 where $a'=m-\left\lfloor\frac{h_1+1}{2}\right\rfloor+1$,
$b'=n-\left\lfloor\frac{h_1+1}{2}\right\rfloor+1$ and $c'=\left\lfloor\frac{h_2+1}{2}\right\rfloor$.
\end{thm}

This paper is organized as follows. In Section 2, we introduce six new families of subgraphs of the grid $B$ and state a theorem for the explicit formulas for the numbers of perfect matchings of these graphs (see Theorem \ref{mainw1}). The theorem is the key result of our paper; and we will prove it in the next three sections. In Section 3, we prove several recurrences for the numbers of perfect matchings of the six families of graphs by using Kuo's graphical condensation method \cite{Kuo}. Then, in Section 4, we show that the formulas in Theorem \ref{mainw1} satisfy the same recurrences obtained in Section 3. This yields an inductive proof of Theorem \ref{mainw1}, which is presented in Section 5. Section 6 is devoted for presenting the proofs of Theorems \ref{Dungeon} and \ref{MCconj} by using the result in Theorem \ref{mainw1}. Finally, we investigate a hidden relation between the graph in Theorem \ref{Dungeon} and the hexagonal dungeon introduced by Blum \cite{Propp} in Section 7.

\section{Six new families of graphs}

 We have investigated various families of subgraphs on the square grid whose perfect matchings are enumerated by perfect powers (see e.g. \cite{Elkies}, \cite{Ciucu}, \cite{Propp}, \cite{Yang}, \cite{Tri2}). However, in most cases, the graphs are either the Aztec rectangles or their variants (including the Aztec diamonds). In this section, we consider six new families of graphs on the square grid that are \textit{not} inspired by the Aztec rectangles. However, their perfect matchings are still enumerated by perfect powers. Strictly speaking, we will define those graphs on the sub-grid  $B$ of the square grid.

\begin{figure}\centering
\includegraphics[width=13cm]{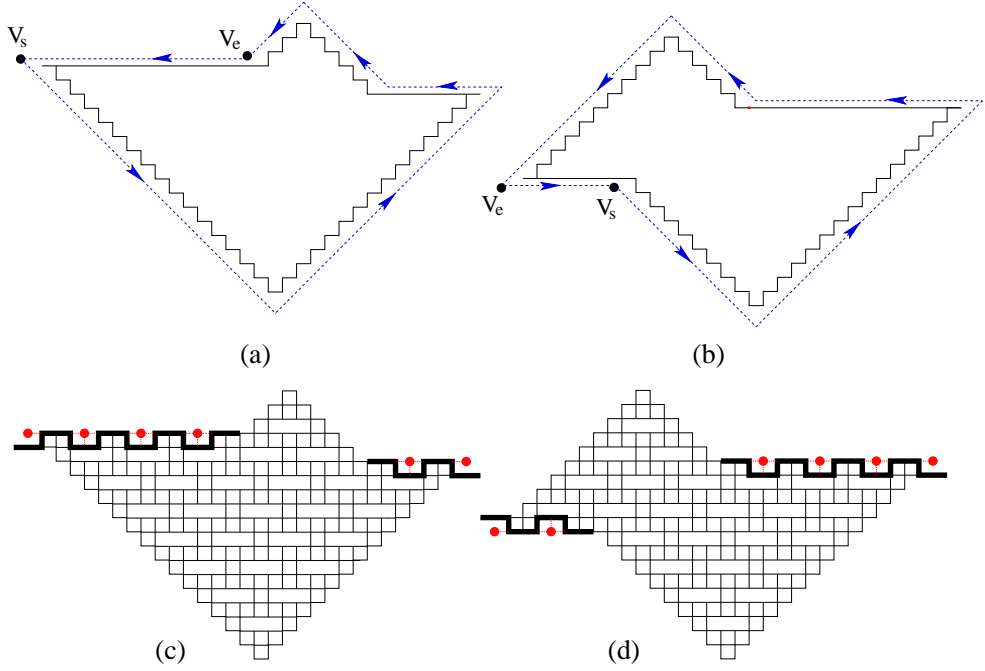}
\caption{Obtaining $F^{(1)}_{9,8,2}$ from $G_{9,8,2}$ (a); obtaining $F^{(1)}_{5,8,4}$ from $G_{5,8,4}$(b); and the graphs $F^{(1)}_{9,8,2}$ (c) and $F^{(1)}_{5,8,4}$ with full details. }
\label{F-1}
\end{figure}
%On the lattice $B$, we call the lattice diagonals obtained by the edges of dotted diamonds covering the cross patterns \textit{stems} of the lattice. The stems fall into two orientations: southeast-to-northwest (SE-NW) or southwest-to-northeast (SW-NE).

Pick the center $V_s$ of a cross pattern on the grid $B$. Let $a,b,c,d,e,x$ be six non-negative integers. We create a six-sided contour starting from $V_s$ as follows. We go $2a\sqrt{2}$ units southeast from $V_s$, then $2b\sqrt{2}$ units northeast,  $4c$ units west,  $2d\sqrt{2}$ units northwest, and $2e\sqrt{2}$ units southwest. We adjust $e$ so that the ending point $V_e$ of the fifth side are on the same level as $V_s$. Finally, we close the contour by going $x$ units west or east, based on whether $V_e$ is on the east or the west of $V_s$ (see Figures \ref{F-1}(a) and (b)). We denote  by $\mathcal{C}^{(1)}(a,b,c)$ the rsulting contour.

 The above choice of $e$ requires that
\[a+e=b+d.\]
If $V_e$ is $x$ units on the east of $V_s$ (i.e. $a>c+d$), then the closure of the contour yields $4a=x+4d+4c$, so $x=4(a-c-d)$. Similarly, in the case when $V_e$ is $x$ units on the west of $V_s$ (i.e. $a<c+d$), we get $x=4(c+d-a)$. Thus, in all cases, we must have $x=4f$, where $f=|a-c-d|$.

We now consider the subgraph $G_{a,b,c}$ of the grid $B$ induced by the vertices inside the contour (see the graphs restricted by the solid boundaries in Figures \ref{F-1}(a) and (b)). Next, along each horizontal side of the contour, we apply a zigzag cut as in the definition of the trimmed Aztec rectangles in the previous section (illustrated by bold zigzag lines in Figures \ref{F-1}(c) and (d); the shaded vertices indicate the vertices of $G_{a,b,c}$,  which have been removed by the trimming process). In particular, for the $c$-side\footnote{From now on, we usually call the first, the second, $\dots$, the sixth sides of the contour the $a$-, the $b$-, $\dots$, and the $f$-sides. The same terminology will be used for the sides of the graph induced by vertices inside the contour.}, we perform the zigzag cut from right to left and stop when having no room on this side for the next step of the cut. Similarly, we cut along $f$-side from \textit{left to right} and also stop when cannot reach further. Denote by $F^{(1)}_{a,b,c}$ the resulting graph\footnote{We will explain why the resulting graph is determined by only three parameters  $a,b,c$ in the next two paragraphs.} (see Figure \ref{F-1}(c) for the case $a>c+d$, and Figure \ref{F-1}(d) for the case $a\leq c+d$).

It is easy to see that if a bipartite graph $G$ admits a perfect matching, then the numbers of vertices in the two vertex classes of $G$ must be the same. If this condition holds, we say that the graph $G$ is \textit{balanced}.

One readily sees that the balance of $F^{(1)}_{a,b,c}$ requires $d=2b-a-2c$. Moreover, to guarantee that the graph is not empty, we assume in addition $b\geq2$. In summary, we have
\begin{equation}\label{constraint1}
d=2b-a-2c\geq 0,
\end{equation}
\begin{equation}\label{constraint2}
e=b+d-a=3b-2a-2c\geq 0,
\end{equation}
and
\begin{equation}\label{constraint3}
f=|a-c-d|=|2a-2b+c|.
\end{equation}
It means that $d,e,f$ depend on $a,b,c$. This explains why our graph $F^{(1)}_{a,b,c}$ is indeed determined by $a,b,c$ (and so is the contour $\mathcal{C}^{(1)}(a,b,c)$).

Next, we consider a variant $A^{(1)}_{a,b,c}$ of the graph $F^{(1)}_{a,b,c}$ as follows. We remove all vertices along the $a$-, $b$-, $d$-, and $e$-sides of $G_{a,b,c}$ (illustrated by white circles in Figures \ref{A-1}(a) and (b)). We now trim along the $c$- and $f$-sides of the resulting graph in the same way as in the definition of $F^{(1)}_{a,b,c}$ (the vertices of $G_{a,b,c}$, which are removed by the trimming process, are illustrated by shaded points in Figures \ref{A-1}(a) and (b); the edges removed are shown by dotted edges). Figures \ref{A-1}(c) and (d) give examples of the graph $A^{(1)}_{a,b,c}$ in the cases $a>c+d$ and $a\leq c+d$, respectively.

\begin{figure}\centering
\includegraphics[width=13cm]{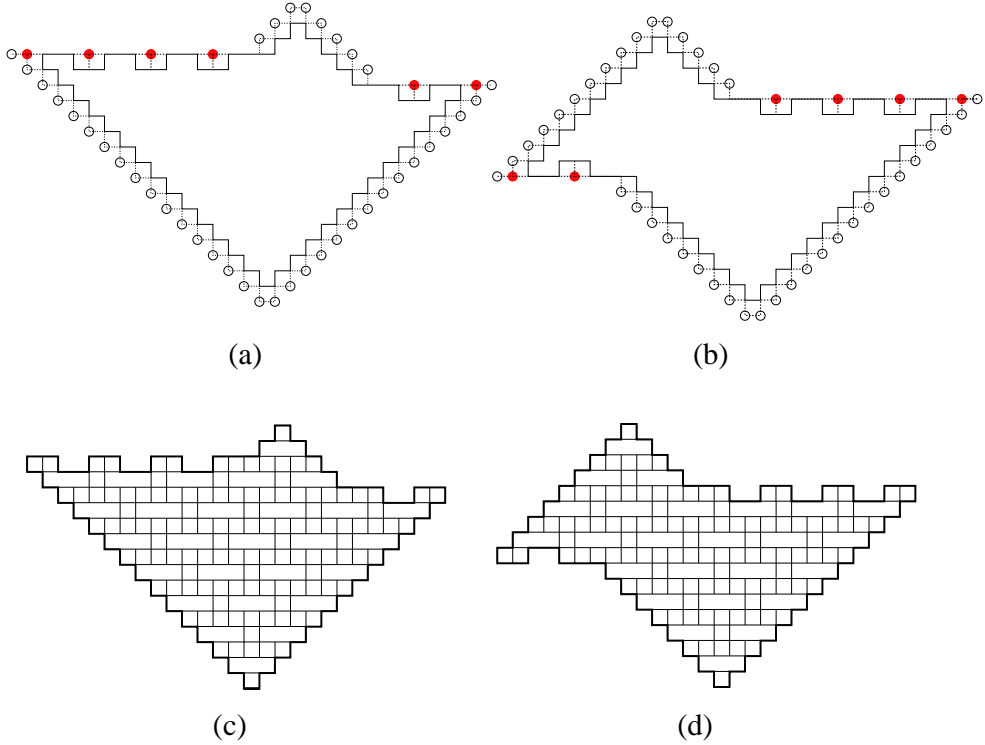}
\caption{Obtaining $A^{(1)}_{9,8,2}$ from $G_{9,8,2}$ (a); obtaining $A^{(1)}_{5,8,4}$ from $G_{5,8,4}$ (b); and the graphs $A^{(1)}_{9,8,2}$ (c) and $A^{(1)}_{5,8,4}$ with full details.}
\label{A-1}
\end{figure}

\medskip

In the definitions of the next families of graphs, we always assume the constraints (\ref{constraint1}), (\ref{constraint2}), (\ref{constraint3}), and $b\geq 2$.

\begin{figure}\centering
\includegraphics[width=13cm]{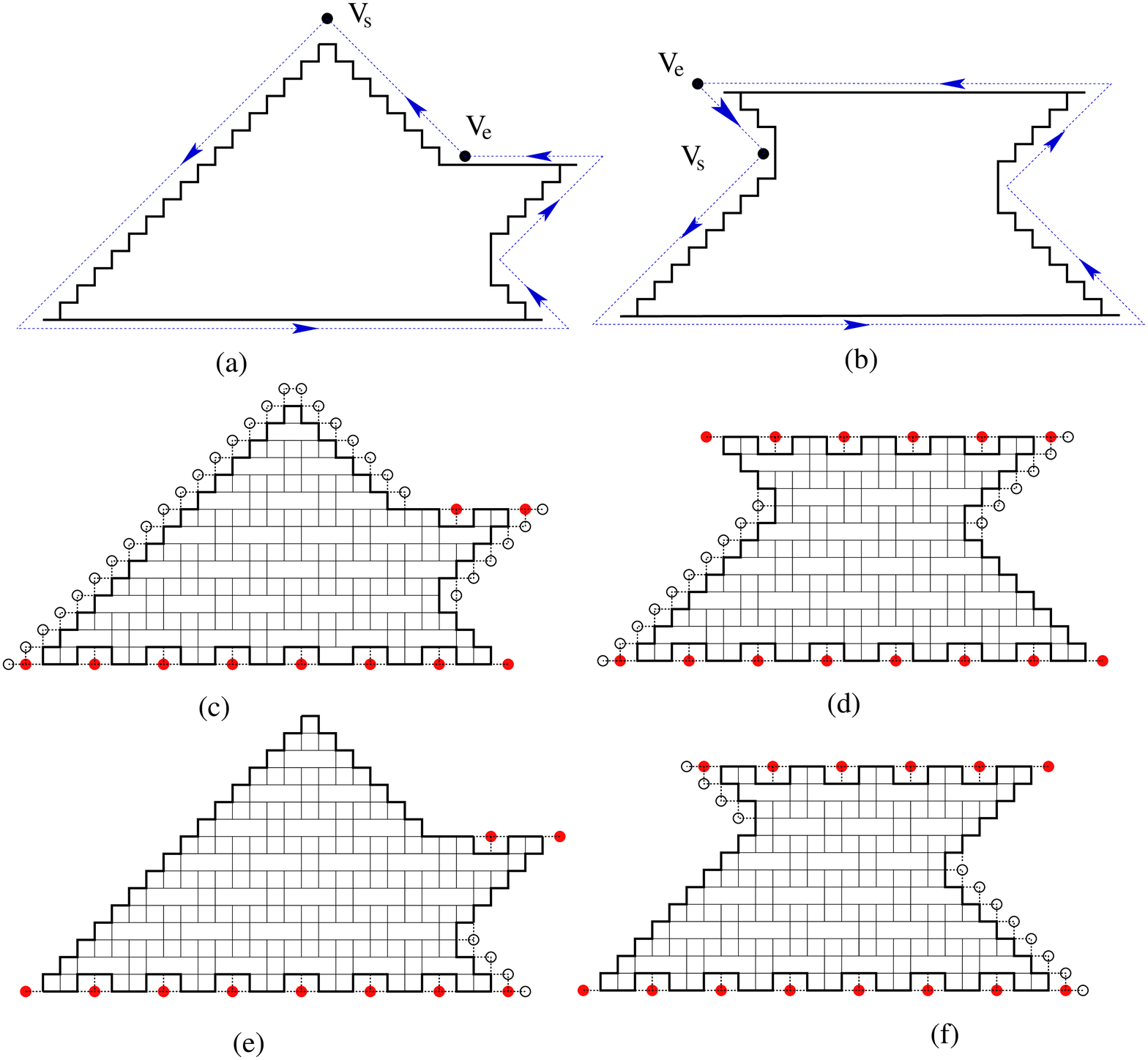}
\caption{The contour $\mathcal{C}^{(2)}_{9,8,2}$ and the graph $H_{9,8,2}$ (a); $\mathcal{C}^{(2)}_{5,8,4}$ and $H_{5,8,4}$ (b); the graphs $F^{(2)}_{9,8,2}$ (c) and $F^{(2)}_{5,8,4}$ (d); and the graphs $A^{(2)}_{9,8,2}$ (e) and $A^{(2)}_{5,8,4}$ (f).}
\label{F-A-2}
\end{figure}

In the spirit of the contour $\mathcal{C}^{(1)}(a,b,c)$, we define a new contour $\mathcal{C}^{(2)}(a,b,c)$ as follows. Starting also from the center  $V_s$ of a cross pattern on the grid $B$, we go $2a\sqrt{2}$ units southwest, then $4b$ units east, $2c\sqrt{2}$ units northwest, $2d\sqrt{2}$ units northeast, $4e$ units west, and $2f\sqrt{2}$ units northwest or southeast, depending on  whether $a>c+d$ or $a\leq c+d$ (see Figures \ref{F-A-2}(a) and (b) for examples). The contour $\mathcal{C}^{(2)}(a,b,c)$ determines a new induced subgraph $H_{a,b,c}$ of the grid $B$ (i.e., $H_{a,b,c}$ is induced by vertices inside the contour). The graphs $H_{9,8,2}$ and $H_{5,8,4}$ are illustrated by the ones restricted by the solid boundaries in Figures \ref{F-A-2}(a) and (b), respectively.

Next, we remove all vertices along the $a$- and $d$-sides of $H_{a,b,c}$; and we also remove the vertices along the $f$-side if $a>c+d$ (see white circles in Figures \ref{F-A-2}(c) and (d)). Similar to the graph $F^{(1)}_{a,b,c}$, we obtain the graph $F^{(2)}_{a,b,c}$ by applying the zigzag cuts along two horizontal sides of the resulting graph (which are now the $b$- and $e$-sides). The graphs $F^{(2)}_{9,8,2}$ and $F^{(2)}_{5,8,4}$ are pictured in Figures \ref{F-A-2}(c) and (d), respectively; the shaded points also indicate the vertices of $H_{a,b,c}$, which are removed by trimming process.

Similarly, the graph $A^{(2)}_{a,b,c}$ is obtained from $H_{a,b,c}$ by  removing vertices along the $c$-side, as well as vertices along the $f$-side when $a\leq c+d$, and trimming along two horizontal sides (illustrated in Figures \ref{F-A-2}(e) and (f)).

\begin{figure}\centering
\includegraphics[width=13cm]{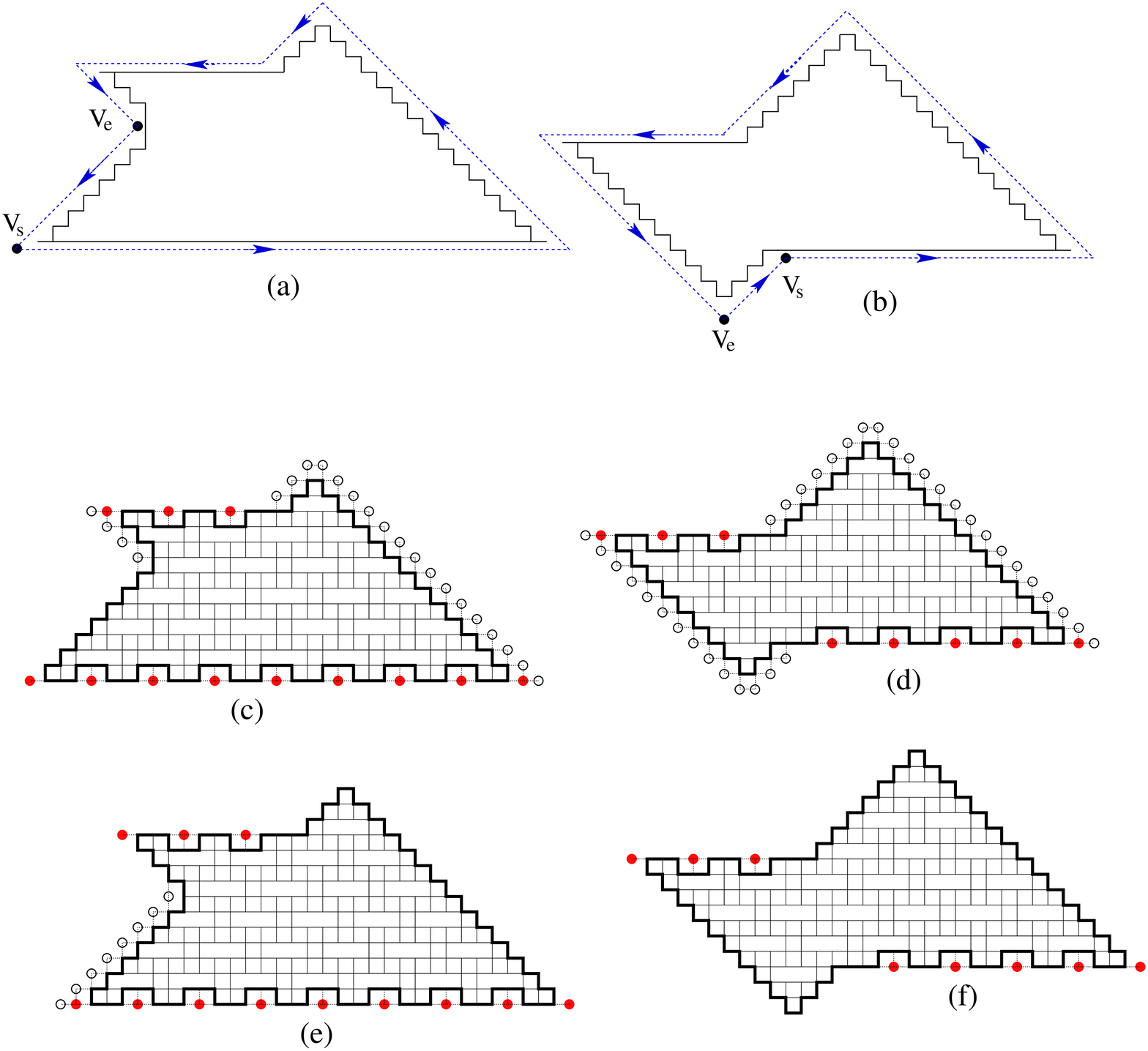}
\caption{The contour $\mathcal{C}^{(3)}_{9,8,2}$ and the graph $H_{9,8,2}$ (a); $\mathcal{C}^{(3)}_{5,8,4}$ and $H_{5,8,4}$ (b); the graphs $F^{(3)}_{9,8,2}$ (c) and $F^{(3)}_{5,8,4}$ (d); and the graphs $A^{(3)}_{9,8,2}$ (e) and $A^{(3)}_{5,8,4}$ (f).}
\label{F-A-3}
\end{figure}

We define next the third contour $\mathcal{C}^{(3)}(a,b,c)$ in the same fashion as the previous ones. Starting the contour at the same point $V_s$, this time we go east, northwest, southwest, west, southeast, and finally southwest or northeast depending on whether $a>c+d$ or $a\leq c+d$. The side-lengths of the contour are now $4a$, $2b\sqrt{2}$, $2c\sqrt{2}$, $4d$, $2e\sqrt{2}$, and $2f\sqrt{2}$, respectively (see Figures \ref{F-A-3}(a) and (b)). We denote by $K_{a,b,c}$ the subgraph of the grid $B$ induced by the vertices inside $\mathcal{C}^{(3)}_{a,b,c}$ (see the graph restricted by the solid contours in Figures \ref{F-A-3}(a) and (b)). We will consider two graphs obtained from $K_{a,b,c}$ by a removing-trimming process in the next paragraph.

The graph $F^{(3)}_{a,b,c}$ is obtained from $K_{a,b,c}$ by removing vertices along its $b$-, $c$-, and $e$-sides, as well as vertices along the $f$-side when $a\leq c+d$, and trimming along two horizontal sides (along the $a$-side from right to left, and the $d$-side from left to right). The graph $A^{(3)}_{a,b,c}$ is obtained similarly by removing vertices along the $f$-side if $a>c+d$, and trimming along the two horizontal sides. Figures \ref{F-A-3}(c), (d), (e), and (f) show the graphs $F^{(3)}_{9,8,2}$, $F^{(3)}_{5,8,4}$, $A^{(3)}_{9,8,2}$, and $A^{(3)}_{5,8,4}$, respectively.

The numbers of perfect matchings of the six graphs $A^{(i)}_{a,b,c}$'s and $F^{(i)}_{a,b,c}$'s are all given by powers of $2$, $5$ and $11$ in the theorem stated below.

\begin{thm}\label{mainw1}
Assume that $a$, $b$ and $c$ are three non-negative integers satisfying $b\geq 2$, $d:=2b-a-2c\geq0$, and $e:=3b-2a-2c\geq0$. Then
\begin{equation}\label{w1eq1}
\M\left(A^{(1)}_{a,b,c}\right)=\alpha(a,b,c) 2^{g(a,b,c+1)}5^{g(a,b,c)}11^{q(a,b,c)},
\end{equation}
\begin{equation}\label{w1eq2}
\M\left(A^{(2)}_{a,b,c}\right)=\alpha(a,b,c) 2^{g(a,b,c-1)-\lfloor(a-c+1)/3\rfloor+(a-b)}5^{g(a,b,c)}11^{q(a,b,c)},
\end{equation}
\begin{equation}\label{w1eq3}
\M\left(A^{(3)}_{a,b,c}\right)=\alpha(a,b,c) 2^{g(a,b,c-1)-\lfloor(a-c+1)/3\rfloor}5^{g(a,b,c)}11^{q(a,b,c)},
\end{equation}
\begin{equation}\label{w1eq4}
\M\left(F^{(1)}_{a,b,c}\right)=\beta(a,b,c)2^{g(a,b,c-1)}5^{g(a,b,c)}11^{q(a,b,c)},
\end{equation}
\begin{equation}\label{w1eq5}
\M\left(F^{(2)}_{a,b,c}\right)=\beta(a,b,c)2^{g(a,b,c+1) +\lfloor(a-c+1)/3\rfloor-(a-b)}5^{g(a,b,c)}11^{q(a,b,c)},
\end{equation}
and
\begin{equation}\label{w1eq6}
\M\left(F^{(3)}_{a,b,c}\right)=\beta(a,b,c)2^{g(a,b,c+1) +\lfloor(a-c+1)/3\rfloor}5^{g(a,b,c)}11^{q(a,b,c)},
\end{equation}
 where $\alpha(a,b,c)$, $\beta(a,b,c)$, $q(a,b,c)$ and $g(a,b,c)$ are defined as in Theorem \ref{MCconj}.
\end{thm}

The proof of Theorem \ref{mainw1} will be given in the next three sections. In particular, Sections 3 and 4 show that the expressions on the left and right
hand sides of each of the equalities (\ref{w1eq1})--(\ref{w1eq6}) both satisfy the same recurrences. Then we will give an inductive proof for the theorem in Section 5.

\section{Recurrences of $\M\left(A^{(i)}_{a,b,c}\right)$ and $\M\left(F^{(i)}_{a,b,c}\right)$}

Eric Kuo (re)proved the Aztec diamond theorem (see \cite{Elkies}) by using a method called ``graphical condensation" (see \cite{Kuo}). The key of his proof is the following combinatorial interpretation of the Desnanot-Jacobi identity in linear algebra (see e.g. \cite{Mui}, pp. 136--149).

\begin{thm} [Kuo's Condensation Theorem \cite{Kuo}] \label{kuothm} Let $G$ be a planar bipartite graph, and $V_1$ and $V_2$ the two vertex classes with$|V_1|=|V_2|$. Assume in addition that $x,y,z$ and $t$ are four vertices appearing in a cyclic order on a face of $G$ so that $x,z\in V_1$ and $y,t\in V_2$. Then
\begin{equation}\label{kuoeq}
\M(G)\M(G-\{x,y,z,t\})=\M(G-\{x,y\})\M(G-\{z,t\})+\M(G-\{t,x\})\M(G-\{y,z\}).
\end{equation}
\end{thm}

In this section, we will use the Kuo's Condensation Theorem to prove that the numbers of perfect matchings of  the six families of graphs  $A^{(i)}_{a,b,c}$'s and $F^{(i)}_{a,b,c}$'s, for $i=1,2,3$, satisfy the following six recurrences (with some constraints).

 We use the notations $\bigstar(a,b,c)$ and $\blacklozenge(a,b,c)$ for general functions from $\mathbb{Z}^3$ to $\mathbb{Z}$. We consider the following six recurrences:
\begin{equation}\tag{R1}
\begin{split}
\bigstar(a,b,c)\bigstar(a-3,b-3,c-2)=\bigstar(a-2,b-1,c)\bigstar(a-1,b-2,c-2)\\+\bigstar(a-1,b-1,c-1)\bigstar(a-2,b-2,c-1),
\end{split} \label{R1}
\end{equation}

\medskip

\begin{equation}\tag{R2}
  \begin{split}
\bigstar(a,b,c)\bigstar(a-2,b-2,c)=\bigstar(a-1,b-1,c)^2\\+\bigstar(a,b,c+1)\bigstar(a-2,b-2,c-1),
  \end{split}\label{R2}
\end{equation}

\medskip

\begin{equation}\tag{R3}
     \begin{split}
      \bigstar(a,b,0)\bigstar(a-2,b-2,0)=\bigstar(a-1,b-1,0)^2\\+\bigstar(a,b,1)\bigstar(3b-2a,2b-a,1),
      \end{split}\label{R3}
\end{equation}

\medskip

\begin{equation}\tag{R4}
  \begin{split}
         \bigstar(a,b,c)\bigstar(a-2,b-3,c-2)=\bigstar(a-1,b-1,c)\bigstar(a-1,b-2,c-2)\\+\bigstar(a-2,b-2,c-1)\bigstar(a,b-1,c-1),
  \end{split}\label{R4}
\end{equation}

\medskip

\begin{equation}\tag{R5}
\begin{cases}
    \begin{split}
       \bigstar(a,b,c)\bigstar(a-2,b-3,c-2)=\blacklozenge(c,b-1,a-1)\bigstar(a-1,b-2,c-2)\\+\bigstar(a-2,b-2,c-1)\bigstar(a,b-1,c-1);
    \end{split}\\

    \begin{split}
      \blacklozenge(a,b,c)\blacklozenge(a-2,b-3,c-2)= \bigstar(c,b-1,a-1)\blacklozenge(a-1,b-2,c-2)\\+\blacklozenge(a-2,b-2,c-1)\blacklozenge(a,b-1,c-1),
    \end{split}
    \end{cases}
    \label{R5}
\end{equation}

and

\begin{equation}\tag{R6}
\begin{cases}
     \begin{split}
      \bigstar(a,b,0)\bigstar(a-2,b-2,0)=\bigstar(a-1,b-1,0)^2\\+\bigstar(a,b,1)\blacklozenge(3b-2a,2b-a,1);
      \end{split}\\
     \begin{split}
      \blacklozenge(a,b,0)\blacklozenge(a-2,b-2,0)=\blacklozenge(a-1,b-1,0)^2\\+\blacklozenge(a,b,1)\bigstar(3b-2a,2b-a,1).
      \end{split}
      \end{cases}
      \label{R6}
\end{equation}

We notice that if $\bigstar \equiv  \blacklozenge$ in the recurrence (\ref{R6}), we get the recurrence (\ref{R3}).

\begin{lem}\label{weight1}
Let $a$, $b$ and $c$ be non-negative integers so that  $b\geq5$,  $c\geq 2$, $d:=2b-a-2c\geq0$, $e:=3b-2a-2c\geq0$, and $a> c+d$. Then for $i=1,2,3$ the numbers of perfect matchings $\M(A^{(i)}_{a,b,c})$ and $\M(F^{(i)}_{a,b,c})$ all satisfy the recurrence $(\ref{R1})$, i.e. we have
\begin{equation}\label{weight1a}
\begin{split}
\M\left(A^{(i)}_{a,b,c}\right)M\left(A^{(i)}_{a-3,b-3,c-2}\right)=\M\left(A^{(i)}_{a-2,b-1,c}\right)\M\left(A^{(i)}_{a-1,b-2,c-2}\right)\\
+\M\left(A^{(i)}_{a-1,b-1,c-1}\right)\M\left(A^{(i)}_{a-2,b-2,c-1}\right)
\end{split}
\end{equation}
and
\begin{equation}\label{weight1b}
\begin{split}
\M\left(F^{(i)}_{a,b,c}\right)M\left(F^{(i)}_{a-3,b-3,c-2}\right)=\M\left(F^{(i)}_{a-2,b-1,c}\right)\M\left(F^{(i)}_{a-1,b-2,c-2}\right)\\
\M\left(F^{(i)}_{a-1,b-1,c-1}\right)\M\left(F^{(i)}_{a-2,b-2,c-1}\right).
\end{split}
\end{equation}
\end{lem}

\begin{figure}\centering
\includegraphics[width=12cm]{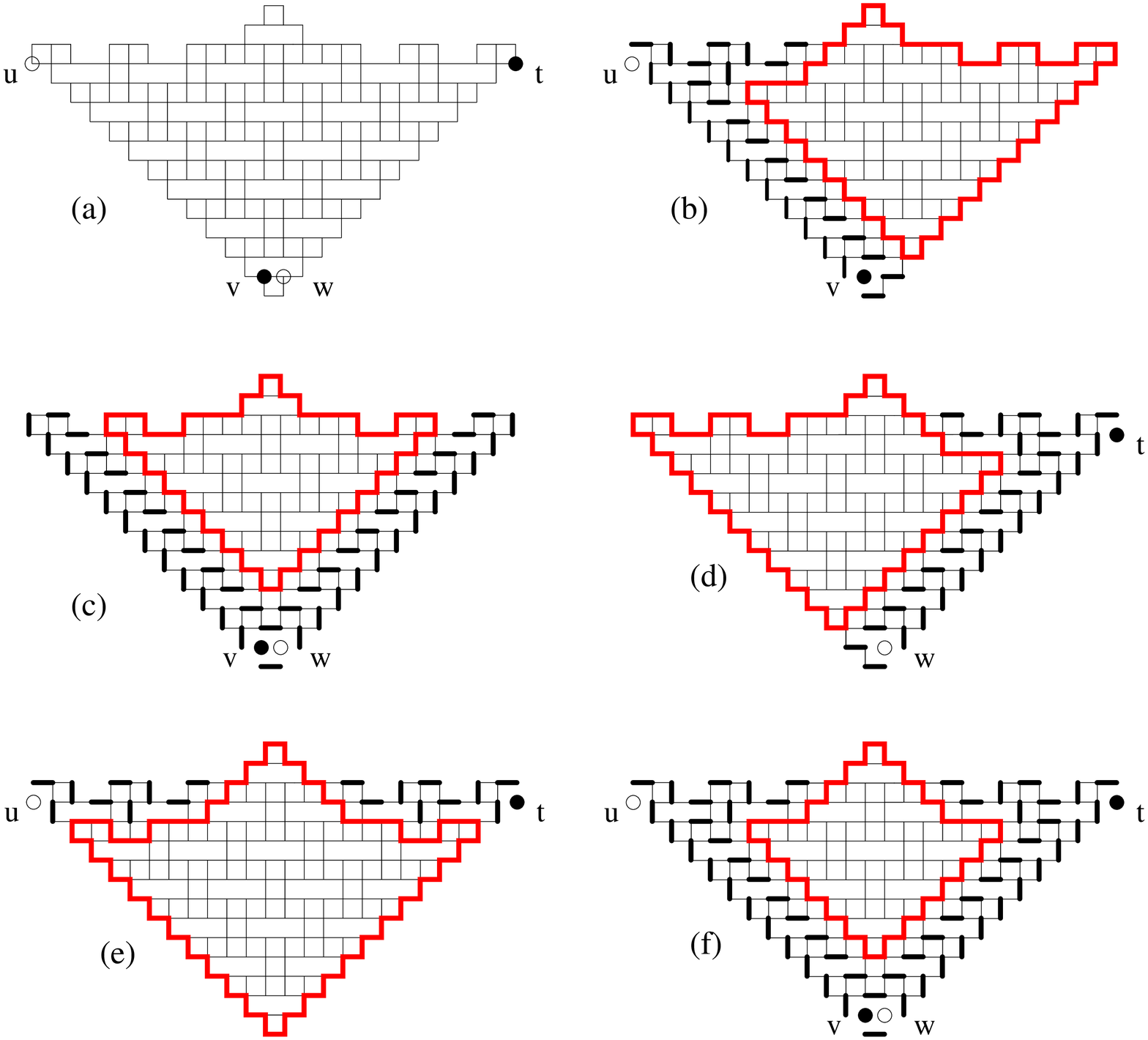}
\caption{Illustrating the proof of  Lemma \ref{weight1}.}
\label{Trimnew4}
\end{figure}

\begin{figure}\centering
\includegraphics[width=12cm]{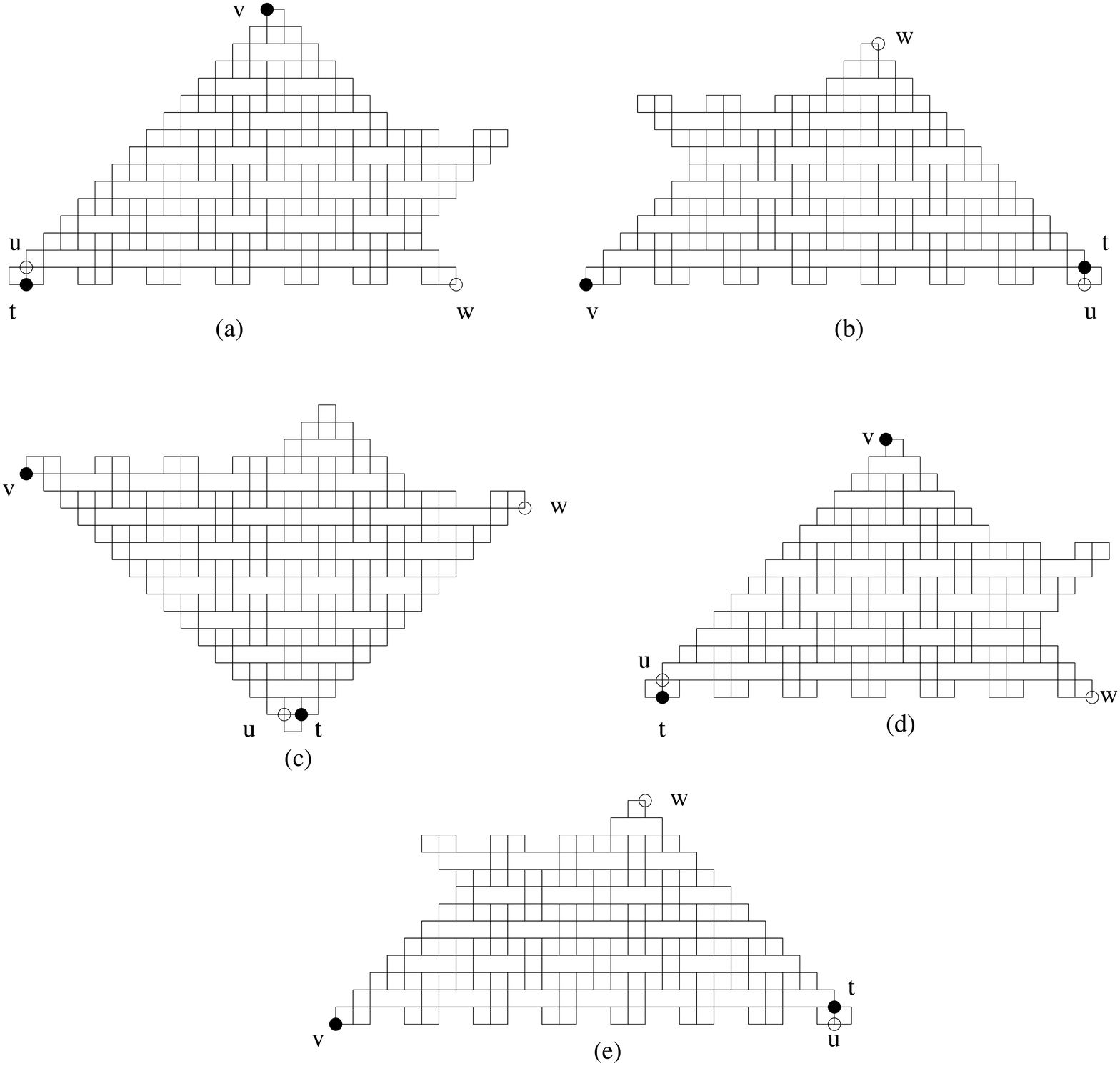}
\caption{Illustrating the proof of  Lemma \ref{weight1}.}
\label{Trimnew5}
\end{figure}

%\begin{figure}\centering
%\includegraphics[width=12cm]{F1A.eps}
%\caption{The graphs $A^{(1)}_{9,8,2}$ (a), $A^{(2)}_{9,8,2}$ (b), and $A^{(3)}_{9,8,2}$ (c).}
%\label{F1A}
%\end{figure}
%\begin{figure}\centering
%\includegraphics[width=12cm]{F1R.eps}
%\caption{The graphs $F^{(1)}_{9,8,2}$ (a), $F^{(2)}_{9,8,2}$ (b), and $F^{(3)}_{9,8,2}$ (c).}
%\label{F1R}
%\end{figure}

%\input{F1A}
%\input{F1R}

\begin{proof} First, we prove the equality (\ref{weight1a}), for $i=1$.

 Apply Kuo's Condensation Theorem \ref{kuothm} to the graphs $G:=A^{(1)}_{a,b,c}$ with the four vertices $u,v,w,t$ chosen as in Figure \ref{Trimnew4}(a) (for $a=b=8$ and $c=3$). In particular, we pick $u$ on the west corner, $v$ and $w$ on the south corner, and $t$ on the east corner of the graph. Consider the graph $G-\{u,v\}$. We remove forced edges from $G-\{u,v\}$ (indicated by bold edges in Figure \ref{Trimnew4}(b)) and obtain a graph isomorphic to $A^{(1)}_{a-2,b-1,c}$ (see the graph restricted by the bold contour in Figure \ref{Trimnew4}(b)). In particular, we obtain \[\M(G-\{u,v\})=\M(A^{(1)}_{a-2,b-1,c}).\] Similarly, we get
\[\M(G-\{v,w\})=\M(A^{(1)}_{a-1,b-2,c-2}) \text{ (see Figure  \ref{Trimnew4}(c))},\]
\[\M(G-\{w,t\})=\M(A^{(1)}_{a-1,b-1,c-1}) \text{ (see Figure \ref{Trimnew4}(d))},\]
\[\M(G-\{t,u\})=\M(A^{(1)}_{a-2,b-2,c-1}) \text{ (see Figure \ref{Trimnew4}(e))},\]
and
\[\M(G-\{u,v,w,t\})=\M(A^{(1)}_{a-3,b-3,c-2}) \text{ (see Figure  \ref{Trimnew4}(f))}.\]
Substituting the above five equalities into the equality (\ref{kuoeq}) in Kuo's Condensation Theorem, we get (\ref{weight1a}), for $i=1$.

Similarly, apply Kuo's Condensation Theorem to the graphs $A^{(2)}_{a,b,c}$ and $A^{(3)}_{a,b,c}$ with the four vertices $u,v,w,t$ chosen as in Figures \ref{Trimnew5}(a) and (b), respectively, we get the equality (\ref{weight1a}), for $i=2,3$.

Finally, (\ref{weight1b}) is obtained by repeating the above argument to the graphs $F^{(i)}_{a,b,c}$'s with the four vertices $u,v,w,t$ selected as in Figures \ref{Trimnew5}(c), (d), and (e).
\end{proof}

\begin{lem}\label{weight4}
Let $a$, $b$ and $c$ be non-negative integers satisfying $a\geq 2$, $b\geq4$, $d:=2b-a-2c\geq2$, and $e:=3b-2a-2c\geq2$.

(a) If $c\geq 1$, then $\M(A^{(i)}_{a,b,c})$ and $\M(F^{(i)}_{a,b,c})$ all satisfy the recurrence (\ref{R2}), for $i=1,2,3$.
%\begin{equation}\label{weight5a}
%\M\left(A^{(i)}_{a,b,c}\right)\M\left(A^{(i)}_{a-2,b-2,c}\right)=\M\left(A^{(i)}_{a-1,b-1,c}\right)^2
%+\M\left(A^{(i)}_{a,b,c+1}\right)\M\left(A^{(i)}_{a-2,b-2,c-1}\right)
%\end{equation}
%and
%\begin{equation}\label{weight5a}
%\M\left(F^{(i)}_{a,b,c}\right)\M\left(F^{(i)}_{a-2,b-2,c}\right)=\M\left(F^{(i)}_{a-1,b-1,c}\right)^2
%+\M\left(F^{(i)}_{a,b,c+1}\right)\M\left(F^{(i)}_{a-2,b-2,c-1}\right).
%\end{equation}

(b) If  $c=0$, then $\M(A^{(1)}_{a,b,0})$ and $\M(F^{(1)}_{a,b,0})$ both satisfy the recurrence (\ref{R3}). Moreover, the two pairs of numbers of tilings
$(\M(A^{(i)}_{a,b,0}),\M(A^{(5-i)}_{a,b,0}))$ and $(\M(F^{(i)}_{a,b,0}),\M(F^{(5-i)}_{a,b,0}))$ both satisfy the recurrence (\ref{R6}), for $i=2,3$.
%\begin{equation}\label{weight5a}
%\M\left(A^{(1)}_{a,b,0}\right)\M\left(A^{(1)}_{a-2,b-2,0}\right)=\M\left(A^{(1)}_{a-1,b-1,0}\right)^2
%+\M\left(A^{(1)}_{a,b,1}\right)\M\left(A^{(1)}_{e,d,1}\right);
%\end{equation}
%for $i=2,3$
%\begin{equation}\label{weight5b}
%\M\left(A^{(i)}_{a,b,0}\right)\M\left(A^{(i)}_{a-2,b-2,0}\right)=\M\left(A^{(i)}_{a-1,b-1,0}\right)^2
%+\M\left(A^{(i)}_{a,b,1}\right)\M\left(A^{(5-i)}_{e,d,1}\right);
%\end{equation}
%
%\begin{equation}\label{weight5c}
%\begin{split}
%\M\left(A^{(3)}_{a,b,0}\right)\M\left(A^{(3)}_{a-2,b-2,0}\right)=\M\left(A^{(3)}_{a-1,b-1,0}\right)\M\left(A^{(3)}_{a-1,b-1,0}\right)\\
%+\M\left(A^{(3)}_{a,b,1}\right)\M\left(A^{(2)}_{e,d,1}\right),\end{split}
%\end{equation}
%\begin{equation}\label{weight5d}
%\M\left(F^{(1)}_{a,b,0}\right)\M\left(F^{(1)}_{a-2,b-2,0}\right)=\M\left(F^{(1)}_{a-1,b-1,0}\right)^2
%+\M\left(F^{(1)}_{a,b,1}\right)\M\left(F^{(1)}_{e,d,1}\right);
%\end{equation}
%and for $i=2,3$
%\begin{equation}\label{weight5e}
%\M\left(F^{(i)}_{a,b,0}\right)\M\left(F^{(i)}_{a-2,b-2,0}\right)=\M\left(F^{(i)}_{a-1,b-1,0}\right)^2
%+\M\left(F^{(i)}_{a,b,1}\right)\M\left(F^{(5-i)}_{e,d,1}\right).
%\end{equation}
%\begin{equation}\label{weight5f}
%\begin{split}
%\M\left(F^{(3)}_{a,b,0}\right)\M\left(F^{(3)}_{a-2,b-2,0}\right)=\M\left(F^{(3)}_{a-1,b-1,0}\right)\M\left(F^{(3)}_{a-1,b-1,0}\right)\\
%+\M\left(F^{(3)}_{a,b,1}\right)\M\left(F^{(2)}_{e,d,1}\right).\end{split}
%\end{equation}
\end{lem}

\begin{figure}\centering
\includegraphics[width=12cm]{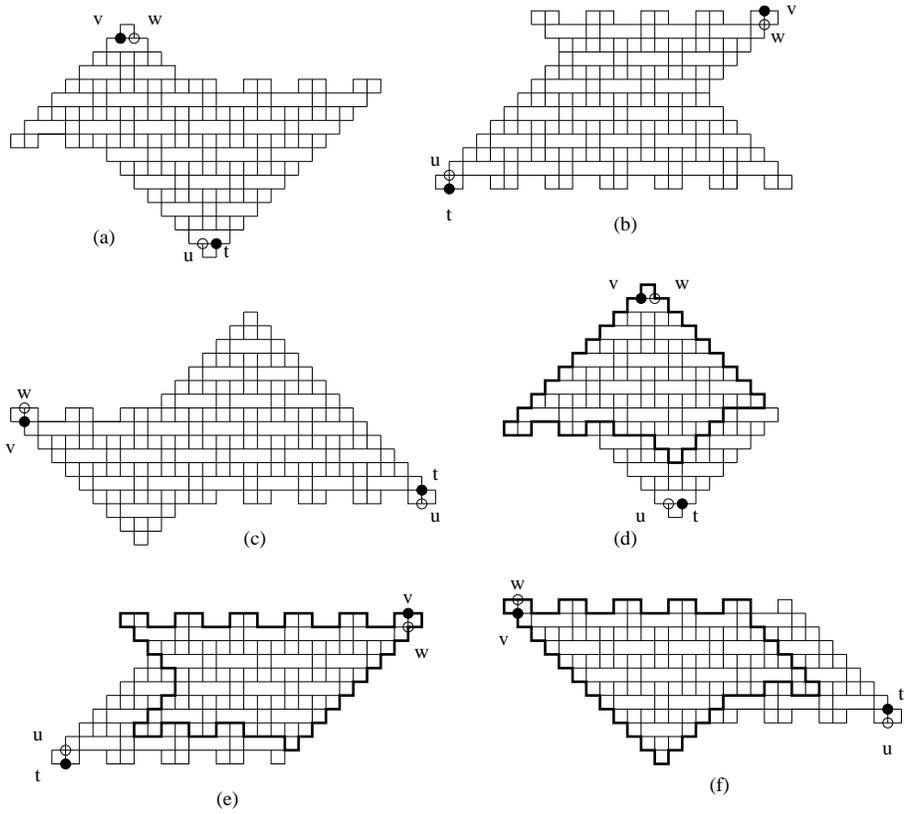}
\caption{he graphs $A^{(1)}_{5,8,4}$ (a), $A^{(2)}_{5,8,4}$ (b), $A^{(3)}_{5,8,4}$ (c), $A^{(1)}_{4,5,0}$ (d), $A^{(2)}_{4,5,0}$ (e), and $A^{(3)}_{4,5,0}$ (f) with the four vertices $u,v,w,t$ in Lemma \ref{weight4}.}
\label{Trimnew6}
\end{figure}

%\begin{figure}\centering
%\includegraphics[width=12cm]{F4A.eps}
%\caption{The graphs $A^{(1)}_{5,8,4}$ (a), $A^{(2)}_{5,8,4}$ (b), and $A^{(3)}_{5,8,4}$ (c).}
%\label{F4A}
%\end{figure}
%\begin{figure}\centering
%\includegraphics[width=12cm]{F4R.eps}
%\caption{The graphs $F^{(1)}_{5,8,4}$ (a), $F^{(2)}_{5,8,4}$ (b), and $F^{(3)}_{5,8,4}$(c).}
%\label{F4R}
%\end{figure}

%\input{F4A}
%\input{F4R}
\begin{proof}
(a) We prove only the statement for the $A$-graphs, as the statement for $F$-graphs can be obtained similarly.
 We need to show for $i=1,2,3$ that
\begin{equation}\label{weight5a}
\M\left(A^{(i)}_{a,b,c}\right)\M\left(A^{(i)}_{a-2,b-2,c}\right)=\M\left(A^{(i)}_{a-1,b-1,c}\right)^2
+\M\left(A^{(i)}_{a,b,c+1}\right)\M\left(A^{(i)}_{a-2,b-2,c-1}\right).
\end{equation}
%and
%\begin{equation}\label{weight5b}
%\M\left(F^{(i)}_{a,b,c}\right)\M\left(F^{(i)}_{a-2,b-2,c}\right)=\M\left(F^{(i)}_{a-1,b-1,c}\right)^2
%+\M\left(F^{(i)}_{a,b,c+1}\right)\M\left(F^{(i)}_{a-2,b-2,c-1}\right).
%\end{equation}

Consider the graphs $A^{(i)}_{a,b,c}$ with the four vertices $u,v,w,t$ located as in Figures \ref{Trimnew6}(a), (b), and (c). Similar to the proof of Lemma \ref{weight1}, we have
\begin{equation}\label{lem2eq1}
M(G-\{u,v\})=\M(A^{(1)}_{a-1,b-1,c}),
\end{equation}
\begin{equation}\label{lem2eq2}
\M(G-\{v,w\})=\M(A^{(1)}_{a-1,b-1,c}),
\end{equation}
\begin{equation}\label{lem2eq3}
\M(G-\{w,t\})=\M(A^{(1)}_{a,b,c+1}),
\end{equation}
\begin{equation}\label{lem2eq4}
\M(G-\{t,u\})=\M(A^{(1)}_{a-2,b-2,c-1}),
\end{equation}
\begin{equation}\label{lem2eq5}
\M(G-\{u,v,w,t\})=\M(A^{(1)}_{a-2,b-2,c}).
\end{equation}
Again, by Kuo's Theorem, we get (\ref{weight5a}).

%The equality (\ref{weight5b}) is obtained by applying analogously the Kuo's Theorem to the graphs $F^{(i)}_{a,b,c}$ with the four vertices $u,v,w,t$ chosen as in Figure \ref{F4R}.

(b)  We prove the statement for the graphs $A^{(i)}_{a,b,c}$'s, and the one for $F^{(i)}_{a,b,c}$'s can be obtained in an analogous manner. In particular, we need to show that
\begin{equation}\label{weight5c}
\M\left(A^{(1)}_{a,b,0}\right)\M\left(A^{(1)}_{a-2,b-2,0}\right)=\M\left(A^{(1)}_{a-1,b-1,0}\right)^2
+\M\left(A^{(1)}_{a,b,1}\right)\M\left(A^{(1)}_{e,d,1}\right),
\end{equation}
and that for $i=2,3$
\begin{equation}\label{weight5d}
\M\left(A^{(i)}_{a,b,0}\right)\M\left(A^{(i)}_{a-2,b-2,0}\right)=\M\left(A^{(i)}_{a-1,b-1,0}\right)^2
+\M\left(A^{(i)}_{a,b,1}\right)\M\left(A^{(5-i)}_{e,d,1}\right).
\end{equation}

The equalities (\ref{weight5c}) and (\ref{weight5d}) can be treated similarly to (\ref{weight5a}). We still pick the four vertices $u,v,w,t$ in the graphs $A^{(i)}_{a,b,0}$'s as in Figures \ref{Trimnew6}(d), (e), and (f). We still have the equalities (\ref{lem2eq1}), (\ref{lem2eq2}),(\ref{lem2eq3}), and (\ref{lem2eq5}), for $c=0$. However, the graph obtained by removing forced edges from the graph $A^{(i)}_{a,b,c}-\{u,t\}$ is not $A^{(i)}_{a-2,b-2,c-1}$ any more (the latter graph is \textit{not} defined when $c=0$); and it is now isomorphic to $A^{(1)}_{e,d,1}$ (resp., $A^{(3)}_{e,d,1}$,  $A^{(2)}_{e,d,1}$), where $d=2b-a$ and $e=3b-2a$.  The graphs  $A^{(1)}_{e,d,1}$, $A^{(3)}_{e,d,1}$,  $A^{(2)}_{e,d,1}$ are illustrated by the ones restricted by the bold contours in Figures \ref{Trimnew6}(d), (e), and (f), respectively. Thus, we have
\begin{equation}
\M(A^{(1)}_{a,b,0}-\{t,u\})=\M(A^{(1)}_{e,d,1}),
\end{equation}
\begin{equation}
\M(A^{(2)}_{a,b,0}-\{t,u\})=\M(A^{(3)}_{e,d,1}),
\end{equation}
and
\begin{equation}
\M(A^{(3)}_{a,b,0}-\{t,u\})=\M(A^{(2)}_{e,d,1}).
\end{equation}
Then (\ref{weight5c}) and (\ref{weight5d}) follow from Theorem \ref{kuothm}.
%The rest of part (b) is obtained by doing similarly to the graphs $F^{(i)}_{a,b,0}$'s, based on Figure \ref{F5R}.
\end{proof}

%\input{F5A}
%\input{F5R}
%\begin{figure}\centering
%\includegraphics[width=12cm]{F5A.eps}
%\caption{The graphs $A^{(1)}_{4,5,0}$ (a), $A^{(2)}_{4,5,0}$ (b), and $A^{(3)}_{4,5,0}$ (c).}
%\label{F5A}
%\end{figure}

%\begin{figure}\centering
%\includegraphics[width=12cm]{F5R.eps}
%\caption{The graphs $F^{(1)}_{4,5,0}$ (a), $F^{(2)}_{4,5,0}$ (b), and $F^{(3)}_{4,5,0}$ (c).}
%\label{F5R}
%\end{figure}

\begin{lem}\label{weight6}
Assume that $a,b,c$ are three non-negative integers satisfying $a\geq2$, $b\geq5$, $c\geq2$, $d:=2b-a-2c\geq0$, and $e:=3b-2a-2c\geq0$. Assume in addition that $a\leq c+d$.

(a) If $d\geq 1$, then  for $i=1,2,3$ the numbers perfect matchings $\M(A^{(i)}_{a,b,c})$ and $\M(F^{(i)}_{a,b,c})$ all satisfy the recurrence (\ref{R4}).
%\begin{equation}\label{weight6a}
%\begin{split}
%\M\left(A^{(i)}_{a,b,c}\right)\M\left(A^{(i)}_{a-2,b-3,c-2}\right)=\M\left(A^{(i)}_{a-1,b-1,c}\right)\M\left(A^{(i)}_{a-1,b-2,c-2}\right)\\
%+\M\left(A^{(i)}_{a-2,b-2,c-1}\right)\M\left(A^{(i)}_{a,b-1,c-1}\right)
%\end{split}
%\end{equation}
%and
%\begin{equation}\label{weight6b}
%\begin{split}
%\M\left(F^{(i)}_{a,b,c}\right)\M\left(F^{(i)}_{a-2,b-3,c-2}\right)=\M\left(F^{(i)}_{a-1,b-1,c}\right)\M\left(F^{(i)}_{a-1,b-2,c-2}\right)\\
%+\M\left(F^{(i)}_{a-2,b-2,c-1}\right)\M\left(F^{(i)}_{a,b-1,c-1}\right).
%\end{split}
%\end{equation}

(b) If $d=0$, then for $i=1,2,3$ the pairs of the numbers of perfect matchings $\left(\M(A^{(i)}_{a,b,c}),\M(F^{(4-i)}_{a,b,c})\right)$ and $\left(\M(A^{(i)}_{a,b,c}),\M(F^{(4-i)}_{a,b,c})\right)$ both satisfy the recurrence (\ref{R5}).
%\begin{equation}\label{weight7a}
%\begin{split}
%\M\left(A^{(i)}_{a,b,c}\right)\M\left(A^{(i)}_{a-2,b-3,c-2}\right)=\M\left(F^{(4-i)}_{c,b-1,a-1}\right)\M\left(A^{(i)}_{a-1,b-2,c-2}\right)\\
%+\M\left(A^{(i)}_{a-2,b-2,c-1}\right)\M\left(A^{(i)}_{a,b-1,c-1}\right)\end{split}
%\end{equation}
%and
%\begin{equation}\label{weight7b}
%\begin{split}
%\M\left(F^{(i)}_{a,b,c}\right)\M\left(F^{(i)}_{a-2,b-3,c-2}\right)=\M\left(A^{(4-i)}_{c,b-1,a-1}\right)\M\left(F^{(i)}_{a-1,b-2,c-2}\right)\\
%+\M\left(F^{(i)}_{a-2,b-2,c-1}\right)\M\left(F^{(i)}_{a,b-1,c-1}\right).
%\end{split}
%\end{equation}
\end{lem}

\begin{figure}\centering
\includegraphics[width=12cm]{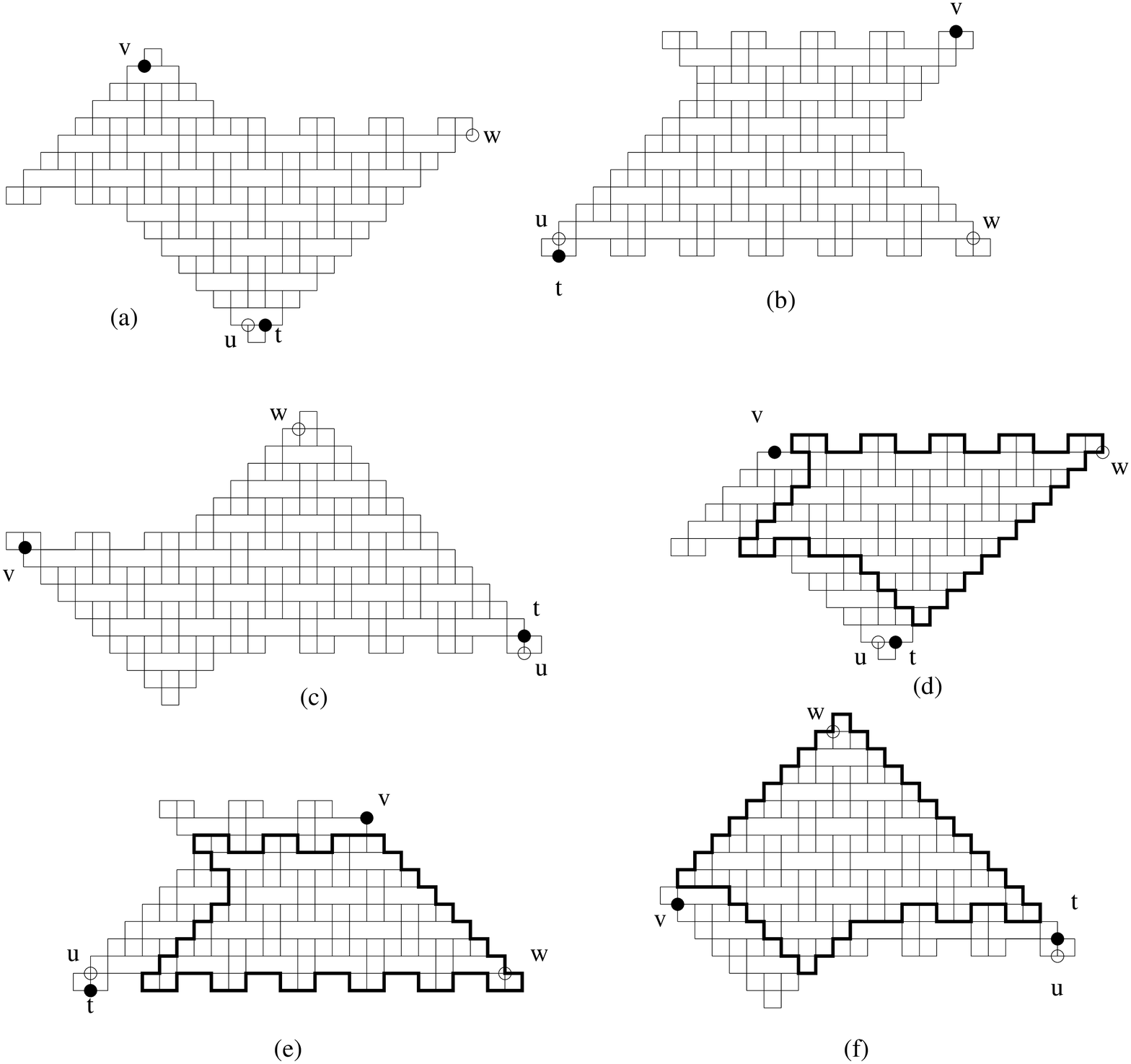}
\caption{The graphs $A^{(1)}_{5,8,4}$ (a), $A^{2)}_{5,8,4}$ (b), $A^{(3)}_{5,8,4}$ (c), $A^{(1)}_{4,8,6}$ (d), $A^{(2)}_{4,8,6}$ (e), and $A^{(3)}_{4,8,6}$ (f) with the selection of the vertices $u,v,w,t$ in Lemma \ref{weight6}.}
\label{Trimnew7}
\end{figure}

%\begin{figure}\centering
%\includegraphics[width=12cm]{F6A.eps}
%\caption{The graphs $A^{(1)}_{5,8,4}$ (a), $A^{2)}_{5,8,4}$ (b), and $A^{(3)}_{5,8,4}$ (c) with the selection of the vertices $u,v,w,t$ in Lemma \ref{weight6}.}
%\label{F6A}
%\end{figure}

%\begin{figure}\centering
%\includegraphics[width=12cm]{F6R.eps}
%\caption{The graphs $F^{(1)}_{5,8,4}$ (a) $F^{(2)}_{5,8,4}$ (b), and $F^{(3)}_{5,8,4}$ (c) with the selection of the vertices $u,v,w,t$ in Lemma \ref{weight6}.}
%\label{F6R}
%\end{figure}

\begin{proof}
(a) We prove here the number of perfect matchings of $A^{(i)}_{a,b,c}$'s satisfy the recurrence (\ref{R4}), and the corresponding statement for $F^{(i)}_{a,b,c}$'s can be obtained in the same fashion.

We need to show that
\begin{equation}\label{weight6a}
\begin{split}
\M\left(A^{(i)}_{a,b,c}\right)\M\left(A^{(i)}_{a-2,b-3,c-2}\right)=\M\left(A^{(i)}_{a-1,b-1,c}\right)\M\left(A^{(i)}_{a-1,b-2,c-2}\right)\\
+\M\left(A^{(i)}_{a-2,b-2,c-1}\right)\M\left(A^{(i)}_{a,b-1,c-1}\right),
\end{split}
\end{equation}
for $i=1,2,3$.

Similar to Lemmas \ref{weight1} and \ref{weight4},  we apply Kuo's Theorem \ref{kuothm} to the graph $A^{(i)}_{a,b,c}$ with the four vertices $u,v,w,t$ chosen as in Figures \ref{Trimnew7}(a), (b), and (c). By considering forced edges, we have the following equalities:
\begin{equation}\label{lem3eq1}
M(G-\{u,v\})=\M(A^{(i)}_{a-1,b-1,c}),
\end{equation}
\begin{equation}\label{lem32eq2}
\M(G-\{v,w\})=\M(A^{(i)}_{a,b-1,c-1}),
\end{equation}
\begin{equation}\label{lem3eq3}
\M(G-\{w,t\})=\M(A^{(i)}_{a-1,b-2,c-2}),
\end{equation}
\begin{equation}\label{lem3eq4}
\M(G-\{t,u\})=\M(A^{(i)}_{a-2,b-2,c-1}),
\end{equation}
\begin{equation}\label{lem3eq5}
\M(G-\{u,v,w,t\})=\M(A^{(i)}_{a-2,b-3,c-2}).
\end{equation}
We get (\ref{weight6a}) by substituting (\ref{lem3eq1})--(\ref{lem3eq5}) into the equality (\ref{kuoeq}) in Kuo's Theorem \ref{kuothm}.

(b) The similarity between the statement for $A$- and the statement for $B$-graphs allows us to prove only first one, as the second one follows similarly.  This aprt can be treated like part (a) by applying Kuo's Theorem \ref{kuothm} to the graphs $A^{(i)}_{a,b,c}$'s  with the four vertices $u,v,w,t$ selected as in Figures \ref{Trimnew7}(d), (e), and (f). The only difference here is that,  after removing forced edges from the graph $A^{(i)}_{a,b,c}-\{u,v\}$, we get the graph $F^{(4-i)}_{c,b-1,a-1}$ (see the graphs restricted by the bold contours in Figures \ref{Trimnew7}(d), (e), and (f)), instead of the graph $A^{(i)}_{a-1,b-1,c}$ as in the part (a).
\end{proof}

Note that if $d:=2b-a-2c=0$ as in the part (b) of Lemma \ref{weight6}, then $2(b-1)-(a-1)-2c=-1$. It means that the graphs  $A^{(i)}_{a-1,b-1,c}$ and $F^{(i)}_{a-1,b-1,c}$ do \textit{not} exist in this case (otherwise the $d$-sides of the contour $C^{(i)}(a,b,c)$ has a negative length, a contradiction).

%\input{F7A}
%\input{F7R}
%\begin{figure}\centering
%\includegraphics[width=12cm]{F7A.eps}
%\caption{The graphs $A^{(1)}_{4,8,6}$ (a), $A^{(2)}_{4,8,6}$(b), and $A^{(3)}_{4,8,6}$ (c).}
%\label{F7A}
%\end{figure}
%\begin{figure}\centering
%\includegraphics[width=12cm]{F7R.eps}
%\caption{The graphs $F^{(1)}_{4,8,6}$ (a), $F^{(2)}_{4,8,6}$ (b), and $F^{(3)}_{4,8,6}$ (c).}
%\label{F7R}
%\end{figure}

\bigskip

For $i=1,2,3$ denote by $\Phi_i(a,b,c)$ the product on the right hand sides of the equalities (\ref{w1eq1})--(\ref{w1eq3}) in Theorem \ref{mainw1}, respectively, and   $\Psi_{i}(a,b,c)$ the product on the right hand sides of equalities (\ref{w1eq4})--(\ref{w1eq6}), respectively.  In the next section, we will show that these functions satisfy the same recurrences (\ref{R1})--(\ref{R6}).

\section{Recurrences for the functions $\Phi_i(a,b,c)$ and $\Psi_i(a,b,c)$ }

%For $i=1,2,3$ we denote by $k_i(a,b,c)$ the exponent of 2 in $\Phi_i(a,b,c)$, and $k_{3+i}(a,b,c)$ the exponent of $2$ in $\Psi_{i}(a,b,c)$. In particular, $k_1(a,b,c)=g(a,b,c+1)$, $k_2(a,b,c)=g(a,b,c-1)-\lfloor(a-c+1)/3\rfloor+(a-b)$, $k_3(a,b,c)=g(a,b,c-1)-\lfloor(a-c+1)/3\rfloor$, $k_4(a,b,c)=g(a,b,c-1)$, $k_5(a,b,c)=g(a,b,c+1) +\lfloor(a-c+1)/3\rfloor-(a-b)$ and $k_6(a,b,c)=g(a,b,c+1) +\lfloor(a-c+1)/3\rfloor$.

\medskip

\begin{lem}\label{reweight1}
For any integers $a,b,c$, and $i=1,2,3$, the functions $\Phi_{i}(a,b,c)$ and $\Psi_i(a,b,c)$ all satisfy the recurrence (\ref{R1}).
%\begin{equation}\label{reweight1eq}
%\begin{split}
%\Psi_i(a,b,c)\Psi_i(a-3,b-3,c-2)=\Psi_i(a-2,b-1,c)\Psi_i(a-1,b-2,c-2)\\+\Psi_i(a-1,b-1,c-1)\Psi_i(a-2,b-2,c-1).
%\end{split}
%\end{equation}
\end{lem}

\begin{proof}
We need to show that
\begin{equation}\label{reweight1eqa}
\begin{split}
\Phi_i(a,b,c)\Phi_i(a-3,b-3,c-2)=\Phi_i(a-2,b-1,c)\Phi_i(a-1,b-2,c-2)\\+\Phi_i(a-1,b-1,c-1)\Phi_i(a-2,b-2,c-1)
\end{split}
\end{equation}
and
\begin{equation}\label{reweight1eqb}
\begin{split}
\Psi_i(a,b,c)\Psi_i(a-3,b-3,c-2)=\Psi_i(a-2,b-1,c)\Psi_i(a-1,b-2,c-2)\\+\Psi_i(a-1,b-1,c-1)\Psi_i(a-2,b-2,c-1),
\end{split}
\end{equation}
for $i=1,2,3$.

%By the definitions of the functions $g(a,b,c)$, $h(a,b,c)$ and $k_i(a,b,c)$, one readily verifies the following facts\footnote{The facts (i), (ii), (iii) and (iv) were first shown in \cite{CL}, Lemma 5.1.}:
%
% (i) If $a-b+c$ is even, then
%\begin{equation}\label{fact1-1}
%\begin{split}
%q(a,b,c)+q(a-3,b-3,c-2)=q(a-2,b-1,c)+q(a-1,b-2,c-2)+1\\=q(a-1,b-1,c-1)+q(a-2,b-2,c-1)+1.
%\end{split}
%\end{equation}
%
% (ii) If $a-b+c$ is odd, then
%\begin{equation}\label{fact1-2}
%\begin{split}
%q(a,b,c)+q(a-3,b-3,c-2)=q(a-2,b-1,c)+q(a-1,b-2,c-2)\\=q(a-1,b-1,c-1)+q(a-2,b-2,c-1).
%\end{split}
%\end{equation}
%
% (iii) If $a-c=3l$ or $3l+1$, then
%\begin{equation}\label{fact1-3}
%\begin{split}
%g(a,b,c)+g(a-3,b-3,c-2)=g(a-2,b-1,c)+g(a-1,b-2,c-2)\\=g(a-1,b-1,c-1)+g(a-2,b-2,c-1).
%\end{split}
%\end{equation}
%
% (iv) If $a-c=3l+2$, then
%\begin{equation}\label{fact1-4}
%\begin{split}
%g(a,b,c)+g(a-3,b-3,c-2)+1=g(a-2,b-1,c)+g(a-1,b-2,c-2)\\=g(a-1,b-1,c-1)+g(a-2,b-2,c-1)+1.
%\end{split}
%\end{equation}
%
%(v) If $a-c\equiv 1$ or $2 \pmod{3}$, then for $i=1,2,\dotsc,6$
%\begin{equation}\begin{split}
%k_i(a,b,c)+k_i(a-3,b-3,c-2)=k_i(a-2,b-1,c)+k_i(a-1,b-2,c-2)\\=k_i(a-1,b-1,c-1)+k_i(a-2,b-2,c-1).\end{split}
%\end{equation}
%
%(vi) If $a-c$ is a multiple of 3, then for $i=1,2,\dotsc,6$
%\begin{equation}\begin{split}
%k_i(a,b,c)+k_i(a-3,b-3,c-2)+1=k_i(a-2,b-1,c)+k_i(a-1,b-2,c-2)\\=k_i(a-1,b-1,c-1)+k_i(a-2,b-2,c-1)+1.\end{split}
%\end{equation}

We first consider the case of even $b$. There are six subcases to distinguish, based on the values of $a-c \pmod{6}$. We show in details here the subcase when $a-c\equiv 0\pmod{6}$ (the other five subcases can be obtained by a perfectly analogous manner).

If $a-c\equiv 0\pmod{6}$, by the definition of functions $G(a,b,c)$ and $q(a,b,c)$, we can cancel out almost all the exponents of 2, 5 and 11 on two sides of the equalities (\ref{reweight1eqa}) and (\ref{reweight1eqb}). The equality (\ref{reweight1eqa}) becomes
\begin{equation}\begin{split}
11\alpha(a,b,c)\alpha(a-3,b-3,c-2)=2\alpha(a-2,b-1,c)\alpha(a-1,b-2,c-2)\\+\alpha(a-1,b-1,c-1)\alpha(a-2,b-2,c-1);\end{split}
\end{equation}
and the equality (\ref{reweight1eqb}) becomes
\begin{equation}\begin{split}
11\beta(a,b,c)\beta(a-3,b-3,c-2)=\beta(a-2,b-1,c)\beta(a-1,b-2,c-2)\\+\beta(a-1,b-1,c-1)\beta(a-2,b-2,c-1).\end{split}
\end{equation}

By definition of the functions $\alpha(a,b,c)$ and $\beta(a,b,c)$, we have here $\alpha(a,b,c)=1$, $\alpha(a-3,b-3,c-2)=1$, $\alpha(a-2,b-1,c)=2$, $\alpha(a-1,b-2,c-2)=2$, $\alpha(a-1,b-1,c-1)=1$, $\alpha(a-2,b-2,c-1)=3$, $\beta(a,b,c)=1$, $\beta(a-3,b-3,c-2)=1$, $\beta(a-2,b-1,c)=2$, $\beta(a-1,b-2,c-2)=3$, $\beta(a-1,b-1,c-1)=1$, and $\beta(a-2,b-2,c-1)=2$. Then the above equalities are equivalent to the following obvious equalities
\[11\cdot 1\cdot 1=2\cdot 2 \cdot 2 +1\cdot 3\]
and
\[11\cdot 1\cdot 1=3\cdot 3+1\cdot 2,\]
respectively.

The remaining case of odd $b$ turns out to follow from the case of even $b$. Indeed,  for $j=0,1,\dotsc,5$, verification of the case of odd $b$ and $a-c \equiv j\pmod{6}$ is the same as the verification of the case of even $b$ and $a-c \equiv 3+j\pmod{6}$.
\end{proof}

\begin{lem}\label{reweight4}
(a) For any integers $a,b,c$, and $i=1,2,3$, the functions $\Phi_{i}(a,b,c)$ and $\Psi_{i}(a,b,c)$ also satisfy the recurrence (\ref{R2}), i.e.
\begin{equation}\label{reweight4eqa}
\begin{split}
\Phi_i(a,b,c)\Phi_i(a-2,b-2,c)=\Phi_i^2(a-1,b-1,c)\\+\Phi_i(a,b,c+1)\Phi_i(a-2,b-2,c-1)
\end{split}
\end{equation}
and
\begin{equation}\label{reweight4eqb}
\begin{split}
\Psi_i(a,b,c)\Psi_i(a-2,b-2,c)=\Psi_i^2(a-1,b-1,c)\\+\Psi_i(a,b,c+1)\Psi_i(a-2,b-2,c-1),
\end{split}
\end{equation}
for $i=1,2,3$.

%\begin{equation}\label{reweight4eq}
%\begin{split}
%\Psi_i(a,b,c)\Psi_i(a-2,b-2,c)=\Psi_i^2(a-1,b-1,c)\\+\Psi_i(a,b,c+1)\Psi_i(a-2,b-2,c-1).
%\end{split}
%\end{equation}

(b) The function $\Phi_1(a,b,c)$ and $\Psi_1(a,b,c)$ both satisfy the recurrence (\ref{R3}), i.e.
\begin{equation}\label{reweight5eqa}
\begin{split}
\Phi_1(a,b,0)\Phi_1(a-2,b-2,0)=\Phi_1^2(a-1,b-1,0)\\+\Phi_1(a,b,1)\Phi_1(3b-2a,2b-a,1)
\end{split}
\end{equation}
and
\begin{equation}\label{reweight5eqb}
\begin{split}
\Psi_1(a,b,0)\Psi_1(a-2,b-2,0)=\Psi_1^2(a-1,b-1,0)\\+\Psi_1(a,b,1)\Psi_1(3b-2a,2b-a,1).
\end{split}
\end{equation}
Moreover, for $i=2,3$, the pairs of functions $(\Phi_{i}(a,b,c), \Phi_{5-i}(a,b,c))$ and $(\Psi_{i}(a,b,c),$ $ \Psi_{5-i}(a,b,c))$ both satisfy the recurrence (\ref{R6}), i.e.
\begin{equation}\label{reweight5eqc}
\begin{split}
\Phi_i(a,b,0)\Phi_i(a-2,b-2,0)=\Phi_i^2(a-1,b-1,0)\\+\Phi_i(a,b,1)\Phi_{5-i}(3b-2a,2b-a,1)
\end{split}
\end{equation}
and
\begin{equation}\label{reweight5eqd}
\begin{split}
\Psi_i(a,b,0)\Psi_i(a-2,b-2,0)=\Psi_i^2(a-1,b-1,0)\\+\Psi_i(a,b,1)\Psi_{5-i}(3b-2a,2b-a,1),
\end{split}
\end{equation}
for $i=2,3$.
\end{lem}

\begin{proof}
(a) %Similar to Lemma \ref{reweight1}, we have the following six facts\footnote{The facts (i), (ii), (iii), and (iv) in this lemma were first introduced in Lemma 52 of \cite{CL}.}:
%
%(i) If $a-b+c$ is even, then
%\begin{equation}\begin{split}
%q(a,b,c)+q(a-2,b-2,c)=2q(a-1,b-1,c)\\=q(a,b,c+1)+q(a-2,b-2,c-1).
%\end{split}
%\end{equation}
%
%(ii) If $a-b+c$ is odd, then
%\begin{equation}\begin{split}
%q(a,b,c)+q(a-2,b-2,c)+1=2q(a-1,b-1,c)+1\\=q(a,b,c+1)+q(a-2,b-2,c-1).
%\end{split}
%\end{equation}
%
%(iii) If $a-c=3l$ or $3l+2$,  then
%\begin{equation}\begin{split}
%g(a,b,c)+g(a-2,b-2,c)=2g(a-1,b-1,c)+1\\=g(a,b,c+1)+g(a-2,b-2,c-1)+1.
%\end{split}
%\end{equation}
%
%(iv) If $a-c=3l+1$, then
%\begin{equation}\begin{split}
%g(a,b,c)+g(a-2,b-2,c)=2g(a-1,b-1,c)\\=g(a,b,c+1)+g(a-2,b-2,c-1).
%\end{split}
%\end{equation}
%
%
%(v) For $i=1,2,3$, if $a-c\equiv 2 \pmod{3}$, then
%\begin{equation}\begin{split}
%k_i(a,b,c)+k_i(a-2,b-2,c)=2k_i(a-1,b-1,c)\\=k_i(a,b,c+1)+k_i(a-2,b-2,c-1),\end{split}
%\end{equation}
%otherwise
%\begin{equation}\begin{split}
%k_i(a,b,c)+k_i(a-2,b-2,c)=2k_i(a-1,b-1,c)+1\\=k_i(a,b,c+1)+k_i(a-2,b-2,c-1)+1.\end{split}
%\end{equation}
%
%(vi) For $i=4,5,6$,  if $a-c\equiv 0 \pmod{3}$, then
%\begin{equation}\begin{split}
%k_i(a,b,c)+k_i(a-2,b-2,c)=2k_i(a-1,b-1,c)\\=k_i(a,b,c+1)+k_i(a-2,b-2,c-1),\end{split}
%\end{equation}
%otherwise
%\begin{equation}\begin{split}
%k_i(a,b,c)+k_i(a-2,b-2,c)=2k_i(a-1,b-1,c)+1\\=k_i(a,b,c+1)+k_i(a-2,b-2,c-1)+1.\end{split}
%\end{equation}

Arguing the same as the proof of Lemma  \ref{reweight1}, we only need to consider for the case of even $b$, and the case of odd $b$ follows.

When $b$ is even, we have also six subcases to distinguish, depending on the values of $a-c \pmod{6}$. Again, we only verify here for the subcase $a-c\equiv 0 \pmod{6}$, and the other subcases can be obtained similarly.

Assume now that $a-c$ is a multiple of $6$. Similar to Lemma \ref{reweight1}, we can cancel out almost all the exponents of $2$, $5$ and $11$ on two sides of  the equalities (\ref{reweight4eqa}) and (\ref{reweight4eqb}). These equalities  are simplified to
\begin{equation}\begin{split}
10\alpha(a,b,c)\alpha(a-2,b-2,c)=\alpha^2(a-1,b-1,c)\\+\alpha(a,b,c+1)\alpha(a-2,b-2,c-1)\end{split}
\end{equation}
 and
\begin{equation}\begin{split}
5\beta(a,b,c)\beta(a-2,b-2,c)=\beta^2(a-1,b-1,c)\\+\beta(a,b,c+1)\beta(a-2,b-2,c-1),\end{split}
\end{equation}
respectively. By the  definition of  $\alpha(a,b,c)$ and $\beta(a,b,c)$, one  can verify easily the above equalities.

(b) We only show that $\Phi_1(a,b,c)$ and $\Psi_1(a,b,c)$ satisfy (\ref{R3}), as the second statement can be proved similarly.

%By the definition, one readily checks
%\begin{equation}\label{refac4a}
%q(a-2,b-2,-1)=q(3b-2a,2b-a,1),
%\end{equation}
%\begin{equation}\label{refac4b}
%\alpha(a-2,b-2,-1)=\alpha(3b-2a,2b-a,1),
%\end{equation}
%and
%\begin{equation}\label{refac4c}
%\beta(a-2,b-2,-1)=\beta(3b-2a,2b-a,1).
%\end{equation}
%
%Moreover, we have for any $t$
%\begin{equation}\label{refac4d}
%g(a-2,b-2,t)=g(3b-2a,2b-a,t+2).
%\end{equation}
%Letting $t=0$ and $-2$ in the fact (\ref{refac4d}), we get for $i=1$ and $4$
%\begin{equation}\label{refac4e}
%k_i(a-2,b-2,-1)=k_i(3b-2a,2b-a,1).
%\end{equation}
%Combining (\ref{refac4a}), (\ref{refac4b}), (\ref{refac4c}),  and (\ref{refac4e}), we have
By definition of functions $\Phi_1$ and $\Psi_1$, one readily verifies that
\begin{equation}
\Phi_1(a-2,b-2,-1)=\Phi_1(3b-2a,2b-a,1)
\end{equation}
and
\begin{equation}
\Psi_1(a-2,b-2,-1)=\Psi_1(3b-2a,2b-a,1),
\end{equation}
then (\ref{reweight5eqa}) and (\ref{reweight5eqb}) follow from part (a), for $c=0$.
%
%Similarly, by (\ref{refac4a}), (\ref{refac4b}), (\ref{refac4c}), and (\ref{refac4d}), we obtain for $i=2,5$
 %\begin{equation}
%\Psi_i(a-2,b-2,-1)=\Psi_{i+1}(3b-2a,2b-a,1),
%\end{equation}
%and for $i=3,6$
 %\begin{equation}
%\Psi_i(a-2,b-2,-1)=\Psi_{i-1}(3b-2a,2b-a,1).
%\end{equation}
%Again, the specialization of part (a) $c=0$ implies (\ref{reweight5eq2}) and (\ref{reweight5eq3}).
\end{proof}

\begin{lem}\label{reweight6}
a) For any integers $a,b,c$ and $i=1,2,3$, the functions $\Phi_{i}(a,b,c)$ and $\Psi_{i}(a,b,c)$ satisfy the recurrence (\ref{R4}), i.e.
\begin{equation}
\begin{split}
\Phi_i(a,b,c)\Phi_i(a-2,b-3,c-2)=\Phi_i(a-1,b-1,c)\Phi_i(a-1,b-2,c-2)\\+\Phi_i(a-2,b-2,c-1)\Phi_i(a,b-1,c-1)
\end{split}
\end{equation}
and
\begin{equation}
\begin{split}
\Psi_i(a,b,c)\Psi_i(a-2,b-3,c-2)=\Psi_i(a-1,b-1,c)\Psi_i(a-1,b-2,c-2)\\+\Psi_i(a-2,b-2,c-1)\Psi_i(a,b-1,c-1),
\end{split}
\end{equation}
for $i=1,2,3$.

(b) For $i=1,2,3$, the pairs of functions $\left(\Phi_{i}(a,b,c), \Psi_{4-i}(a,b,c)\right)$ and $(\Psi_{i}(a,b,c), $ $ \Phi_{4-i}(a,b,c))$ all satisfy the recurrence (\ref{R5}), i.e.
 \begin{equation}
\begin{split}
\Phi_i(a,b,c)\Phi_i(a-2,b-3,c-2)=\Psi_{4-i}(c,b-1,a-1)\Phi_i(a-1,b-2,c-2)\\+\Phi_i(a-2,b-2,c-1)\Phi_i(a,b-1,c-1)
\end{split}
\end{equation}
and
 \begin{equation}
\begin{split}
\Psi_i(a,b,c)\Psi_i(a-2,b-3,c-2)=\Phi_{4-i}(c,b-1,a-1)\Psi_i(a-1,b-2,c-2)\\+\Psi_i(a-2,b-2,c-1)\Psi_i(a,b-1,c-1),
\end{split}
\end{equation}
for $i=1,2,3$.
\end{lem}

\begin{proof}
(a) This part can be treated similarly to Lemma \ref{reweight1} and Lemma \ref{reweight4}(a).

(b) From part (a), we only need to prove  that
\begin{equation}\label{switcha}
\Psi_i(a-1,b-1,c)=\Phi_{4-i}(c,b-1,a-1)
\end{equation}
and
\begin{equation}\label{switchb}
\Phi_i(a-1,b-1,c)=\Psi_{4-i}(c,b-1,a-1),
\end{equation}
for $i=1,2,3$. However, this follows directly from the definition of the functions $\Phi_i$'s and $\Psi_i$'s.
%
%However, one readily verifies the following equalities from definition:
%\begin{equation}\label{exp11}
%q(a-1,b-1,c)=q(c,b-1,a-1),
%\end{equation}
%\begin{equation}\label{exp5}
%g(a-1,b-1,c)=g(c,b-1,a-1),
%\end{equation}
%\begin{equation}\label{exp2}
%k_i(a-1,b-1,c)=k_{7-i}(c,b-1,a-1),
%\end{equation}
%\begin{equation}\label{const1}
%\alpha(a-1,b-1,c)=\beta(c,b-1,a-1)
%\end{equation}
%and
%\begin{equation}\label{const2}
%\beta(a-1,b-1,c)=\alpha(c,b-1,a-1).
%\end{equation}
%
%Substituting (\ref{exp11}), (\ref{exp5}), (\ref{exp2}), (\ref{const1}), and (\ref{const2}) into the formulas of the functions $\Phi_i$ and $\Psi_i$, we get (\ref{switcha}) and (\ref{switchb}).
\end{proof}

\section{Proof of Theorem \ref{mainw1}}

We will prove Theorem \ref{mainw1} in the same fashion as the proof of Theorem 3.1 in \cite{CL}.

\begin{proof}[Proof of Theorem \ref{mainw1}]
We define the a function $P(a,b,c)$ by setting \[P(a,b,c):=a+b+c+d+e+f,\] where $d:=2b-a-2c$, $e:=3b-2a-2c$, and $f:=|2a-2b+c|$ as usual.  Moreover, one readily sees that $P(a,b,c)$ equals $4b-2c$ if $a>c+d$, and $8b-4a-4c$ if $a\leq c+d$. In particular, $P(a,b,c)$ is always even. We call $P(a,b,c)$ the \textit{perimeter} of our six graphs $A^{(i)}_{a,b,c}$'s and $F^{(i)}_{a,b,c}$'s.

We need to show that
\begin{equation}\label{w1eqm}
\M(A^{(i)}_{a,b,c})=\Phi_i(a,b,c) \text{ and } \M(F^{(i)}_{a,b,c})=\Psi_{i}(a,b,c),
\end{equation}
for $i=1,2,3$, by induction of the value the perimeter $P(a,b,c)$ of the six graphs.

For the base cases, one can verify easily  (\ref{w1eqm}) with the help of \texttt{vaxmacs}, a software written by David Wilson\footnote{The software can be downloaded on the link \texttt{http://dbwilson.com/vaxmacs/}.}, for all the triples (a,b,c) satisfying \textit{at least one} of the following conditions:
\begin{enumerate}
\item[(i)] $P(a,b,c)\leq 14$;
\item[(ii)] $b\leq 4$;
\item[(iii)] $c+d=2b-a-c\leq 2$.
\end{enumerate}

For the induction step, we assume that (\ref{w1eqm}) holds for all graphs $A^{(i)}_{a,b,c}$'s and $B^{(i)}_{a,b,c}$  having perimeter $P(a,b,c)$ less than $p$, for some $p\geq 16$. We will prove (\ref{w1eqm}) for all $A$- and $F$-graphs with perimeter $p$.

By the base cases, we only need to show (\ref{w1eqm}) for all graphs having the triple $(a,b,c)$ in the following domain
\begin{equation}
\mathcal{D}:=\{(a,b,c)\in \mathbb{Z}^3:\, P(a,b,c)=p,\, b\geq 5,\, c+d \geq 3,\, d\geq0,\, e\geq 0\}.
\end{equation}

First, we partition $\mathcal{D}$ into four subdomains as follows
\[\mathcal{D}_1:=\mathcal{D}\cap \{2\leq a\leq c+d\},\]
\[ \mathcal{D}_2:=\mathcal{D}\cap \{a\leq 1\},\]
\[ \mathcal{D}_3:=\mathcal{D}\cap \{a>c+d,\, e\geq d\},\]
 and
\[\mathcal{D}_4:=\mathcal{D}\cap\{a>c+d,\, e<d\}.\]

Next, we verify that (\ref{w1eqm}) holds in each of the above subdomains.

First, we consider the case $(a,b,c)\in \mathcal{D}_1$.   We divide further $\mathcal{D}_1$ into four subdomains (not necessarily disjoint) by:
\[\mathcal{D}_{1a}:=\mathcal{D}_1\cap\{d\geq 2,\, c\geq 1\},\]
\[\mathcal{D}_{1b}:=\mathcal{D}_1\cap\{d\geq 2,\, c=0\},\]
\[\mathcal{D}_{1c}:=\mathcal{D}_1\cap\{d\geq 1,\, c\geq 2\},\]
and
\[\mathcal{D}_{1d}:=\mathcal{D}_1\cap\{d=0,\, c\geq 2\}.\]

If $(a,b,c)\in \mathcal{D}_{1a}$, then $P(a-2,b-2,c)=p-8$, $P(a-1,b-1,c)=P(a,b,c+1)=P(a-2,b-2,c-1)=p-4$. Thus, by induction hypothesis, we have for $i=1,2,3$
\begin{equation}\label{eqnew1}\M(A^{(i)}_{a-2,b-2,c})=\Phi_{i}(a-2,b-2,c),\end{equation}
\begin{equation}\label{eqnew2}\M(A^{(i)}_{a-1,b-1,c})=\Phi_{i}(a-1,b-1,c),\end{equation}
\begin{equation}\label{eqnew3}\M(A^{(i)}_{a,b,c+1})=\Phi_{i}(a,b,c+1),\end{equation}
\begin{equation}\label{eqnew4}\M(A^{(i)}_{a-2,b-2,c-1})=\Phi_{i}(a-2,b-2,c-1),\end{equation}
\begin{equation}\label{eqnew5}\M(F^{(i)}_{a-2,b-2,c})=\Psi_{i}(a-2,b-2,c),\end{equation}
\begin{equation}\label{eqnew6}\M(F^{(i)}_{a-1,b-1,c})=\Psi_{i}(a-1,b-1,c),\end{equation}
\begin{equation}\label{eqnew7}\M(F^{(i)}_{a,b,c+1})=\Psi_{i}(a,b,c+1),\end{equation}
and
\begin{equation}\label{eqnew8}\M(F^{(i)}_{a-2,b-2,c-1})=\Psi_{i}(a-2,b-2,c-1).\end{equation}
On the other hand, by Lemmas \ref{weight4} and \ref{reweight4}(a), we have $\M(A^{(i)}_{a,b,c})$, $\M(F^{(i)}_{a,b,c})$, $\Phi_{i}(a,b,c)$, and $\Psi_i(a,b,c)$ all satisfy the recurrence (\ref{R2}), for $i=1,2,3$. Therefore, by the above equalities (\ref{eqnew1})--(\ref{eqnew8}), we get $\M(A^{(i)}_{a,b,c})=\Phi_{i}(a,b,c)$ and  $\M(F^{(i)}_{a,b,c})=\Psi_{i}(a,b,c)$, for $i=1,2,3$.

Similarly, if $(a,b,c)\in \mathcal{D}_{1b},$ $ \mathcal{D}_{1c},$ or $\mathcal{D}_{1d}$, we get (\ref{w1eqm}) by using the recurrences (\ref{R3}) and (\ref{R6}) (see Lemmas \ref{weight4}(b) and \ref{reweight4}(b)),  (\ref{R4}) (see Lemmas \ref{weight6}(a) and \ref{reweight6}(a)), or (\ref{R5}) (see Lemmas \ref{weight6}(b) and \ref{reweight6}(b)), respectively. This implies that (\ref{w1eqm}) holds for any triples $(a,b,c)\in \mathcal{D}_1$.

Next, we consider the case $(a,b,c)\in \mathcal{D}_2$ (i.e. we are assuming $a< c+d$). We reflect the graph $A^{(i)}_{a,b,c}$ about a vertical line, and get the graph $F^{(4-i)}_{f,e,d}$, for $i=1,2,3$. Similarly, we get graph  $A^{(4-i)}_{f,e,d}$ by reflecting the graph $F^{(i)}_{a,b,c}$ about a vertical line. This means that
\begin{equation}\label{flipeq0}
\M(A^{(i)}_{a,b,c})=\M(F^{(4-i)}_{f,e,d}) \text{ and } \M(F^{(i)}_{a,b,c})=\M(A^{(4-i)}_{f,e,d}),
\end{equation}
for $i=1,2,3$.
Moreover, we can verify from the definition of the functions $\Phi_i(a,b,c)$ and $\Psi_i(a,b,c)$ that
\begin{equation}\label{flipeq1}
\Phi_{i}(a,b,c)=\Psi_{4-i}(f,e,d) \text{ and } \Psi_{i}(a,b,c)=\Phi_{4-i}(f,e,d),
\end{equation}
for $i=1,2,3$.  By (\ref{flipeq0}) and (\ref{flipeq1}), we only need to show that
\begin{equation}\label{flipeq6}
\M(A^{(i)}_{f,e,d})=\Phi_{i}(f,e,d) \text{ and } \M(F^{(i)}_{f,e,d})=\Psi_{i}(f,e,d),
\end{equation}
then (\ref{w1eqm}) follows.

If $e\leq 4$, then the triple $(f,e,d)$ satisfy the condition (ii) in the base cases, thus (\ref{flipeq6}) follows. Then
 we can assume that $e\geq 5$. We now need to verify that the triple $(f,e,d)$ is in the domain $\mathcal{D}_1$. It is more convenient to re-write the domain $\mathcal{D}_1$ with all constraints in terms of $a,b,c$ as follows:
\begin{align}
\mathcal{D}_1=\{(a,b,c)\in \mathbb{Z}^3:&  \quad P(a,b,c)=p;\, 2\leq a\leq 2b-a-c;\, b\geq 5;\notag\\
& 2b-a-c \geq 3;\, 2b-a-2c\geq0;\, 3b-2a-2c\geq 0\}.
\end{align}
We have in this case $f=c+d-a\geq c+d-1 \geq 2$ (we are assuming that $a\leq 1$ and $c+d\geq 3$). Moreover, the inequality $f\leq 2e-f-d$ is equivalent to $a\geq 0$, so $(f,e,d)$ satisfies the second constraint of the domain $\mathcal{D}_1$.  This implies from the definition of the perimeter that $P(f,e,d)=8e-4f-4d=8b-4a-4c=p$. Next, the third constraint follows from our assumption $e\geq 5$. For the fourth constraint, we have $2e-f-d=2b-a-c=c+d\geq 3$. Finally, we have $2e-f-2d=c\geq 0$ and $3e-2f-2d=b\geq0$, which imply the last two constraints. Thus, the triple $(f,e,d)$ is indeed in $\mathcal{D}_1$. By the case treated above, we have again (\ref{flipeq6}). It means that (\ref{w1eqm}) has just been verified for the case $(a,b,c)\in \mathcal{D}_2$.

The case $(a,b,c)\in \mathcal{D}_3$ can be treated similarly to the case $(a,b,c)\in \mathcal{D}_1$.
We also divide further $\mathcal{D}_{3}$ into three subdomains:
\[\mathcal{D}_{3a}:=\mathcal{D}_3\cap\{c\geq 2\},\]
\[\mathcal{D}_{3b}:=\mathcal{D}_3\cap\{c=1\},\]
 and
 \[\mathcal{D}_{3c}:=\mathcal{D}_3\cap\{c=0\}.\]
If $(a,b,c)\in \mathcal{D}_{3,a}, \mathcal{D}_{3b}$, or $\mathcal{D}_{3c}$, we use the recurrences (\ref{R1}) (see Lemmas \ref{weight1} and \ref{reweight1}), (\ref{R2}) (see Lemmas \ref{weight4}(a) and \ref{reweight4}(a)), or (\ref{R3}) and (\ref{R6}) (see Lemmas \ref{weight4}(b) and  \ref{reweight4}(b)), respectively.

Finally, we consider the case $(a,b,c)\in \mathcal{D}_4$ (i.e., we are assuming that $a>c+d$).  We also reflect the graphs $A^{(i)}_{a,b,c}$'s and $F^{(i)}_{a,b,c}$'s over a horizontal line, and get the reflection  diagram as follows:
\[A^{(1)}_{a,b,c}\rightarrow A^{(1)}_{b,a,f},\]
\[A^{(2)}_{a,b,c}\rightarrow A^{(3)}_{b,a,f},\]
\[A^{(3)}_{a,b,c}\rightarrow A^{(2)}_{b,a,f},\]
\[F^{(1)}_{a,b,c}\rightarrow F^{(1)}_{b,a,f},\]
\[F^{(2)}_{a,b,c}\rightarrow F^{(3)}_{b,a,f},\]
\[F^{(3)}_{a,b,c}\rightarrow F^{(2)}_{b,a,f}.\]
Moreover, one readily gets from the definition of the functions $\Phi_i(a,b,c)$ and $\Psi_i(a,b,c)$ that
\[\Phi_{1}(a,b,c)=\Phi_{1}(b,a,f),\]
\[\Phi_{2}(a,b,c)=\Phi_{3}(b,a,f),\]
\[\Phi_{3}(a,b,c)=\Phi_{2}(b,a,f),\]
\[\Psi_{1}(a,b,c)=\Psi_{1}(b,a,f),\]
\[\Psi_{2}(a,b,c)=\Psi_{3}(b,a,f),\]
and
\[\Psi_{3}(a,b,c)=\Psi_{2}(b,a,f).\]
Therefore, we only need to verify that
\begin{equation}\label{flipeq4}
\M(A^{(i)}_{b,a,f})=\Phi_{i}(b,a,f) \text{ and } \M(F^{(i)}_{b,a,f})=\Psi_{i}(b,a,f),
\end{equation}
for $i=1,2,3$, and (\ref{w1eqm}) follows.

If $a\leq 4$ or $2a-b-f=b-c\leq2$, then  $(b,a,f)$ falls into one of the triples in our base cases, and (\ref{flipeq4}) follows. Therefore we can assume that $a\geq4$ and $b-c\geq 3$. Similar to the case when $(a,b,c)\in \mathcal{D}_2$, if $c\geq 1$, one readily verifies that $(b,a,f)\in\mathcal{D}_3$; and if $c=0$, it is easy to see that $(b,a,f)\in\mathcal{D}_1$. Then, by the cases treated above, we obtain also (\ref{flipeq4}).
\end{proof}

\section{Proofs of Theorems \ref{Dungeon} and \ref{MCconj}}

Before going to the proofs of Theorems \ref{Dungeon} and \ref{MCconj}, we quote the following useful lemma that was proved in \cite{Tri1} (see Lemma 3.6(a)).

\begin{lem}[Graph Splitting Lemma]\label{GS}
Let $G$ be a bipartite graph, and let $V_1$ and $V_2$ be the two vertex classes.% Let $H$ be an induced subgraph of $G$.

Assume that an induced subgraph $H$ of $G$ satisfies following two conditions:
\begin{enumerate}
\item[(i)] \text{\rm{(Separating Condition)}} There are no edges of $G$ connecting a vertex in \\$V(H)\cap V_1$ and a vertex in $V(G-H)$.

\item[(ii)] \text{\rm{(Balancing Condition)}} $|V(H)\cap V_1|=|V(H)\cap V_2|$.
\end{enumerate}
Then
\begin{equation}
\M(G)=\M(H)\, \M(G-H).
\end{equation}
\end{lem}

\begin{figure}\centering
\includegraphics[width=12cm]{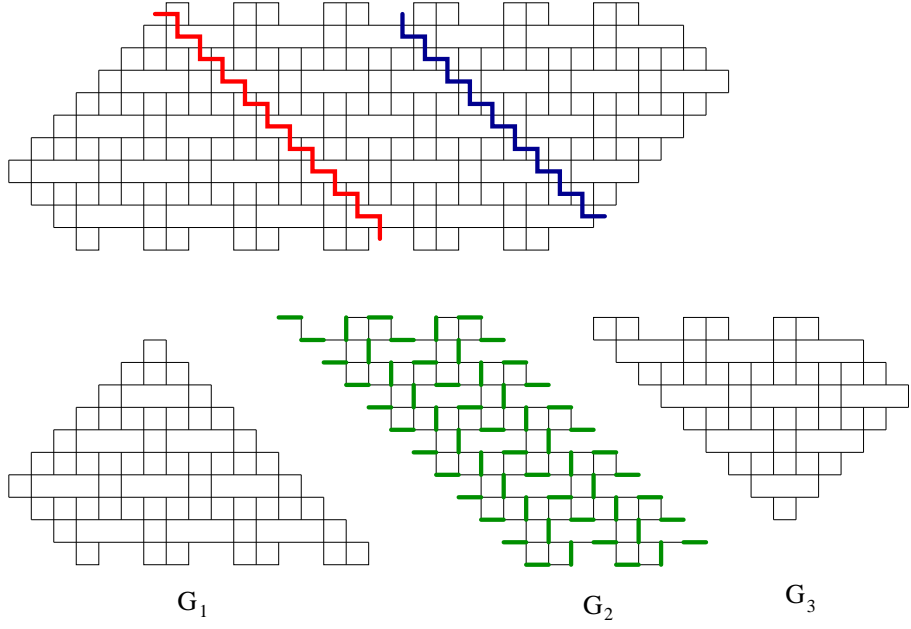}
\caption{Illustrating the proof of Theorem \ref{Dungeon}.}
\label{Trimnew8}
\end{figure}

\begin{proof}[Proof of Theorem \ref{Dungeon}]

We split the graph $TR_{a,b}$ into three subgraphs $G_1$, $G_2$ and $G_3$ by two zigzag cuts as in Figure \ref{Trimnew8}, for $a=2$ and $b=6$.
One readily sees that $G_1$ satisfies the conditions in Graph-splitting Lemma \ref{GS} as an induced  subgraph of $G$, and $G_2$ in turn satisfies the conditions of the lemma as an induced subgraph of $G-G_1$. Therefore, we obtain
\begin{equation}
\M(TR_{a,b})=\M(G_1)\M(G-G_1)=\M(G_1)\M(G_2)\M(G_3).
\end{equation}
It is easy to see that $G_2$ has a unique perfect matching (see the bold edges in Figure \ref{Trimnew8}), and the graph $G_1$ and $G_2$ are isomorphic to $A^{(3)}_{2a,3a,2a}$ and $F^{(1)}_{2a,3a,2a}$, respectively. By Theorem \ref{mainw1}, we obtain
\begin{align}
\M(G_{a,2a,b})&=\M(A^{(3)}_{2a,3a,2a})\M(F^{(3)}_{2a,3a,2a})\notag\\
&=2^{g(2a,3a,2a-1)+g(2a,3a,2a+1)}5^{2g(2a,3a,2a)}11^{2q(2a,3a,2a)}\notag\\
&=2^{a(a+1)+a(a-1)}5^{2a^2}11^{2\lfloor \frac{a^2}{4}\rfloor}\notag\\
&=10^{2a^2}11^{\lfloor \frac{a^2}{2}\rfloor},
\end{align}
then the theorem follows.
\end{proof}

\bigskip

\begin{proof} [Proof of Theorem \ref{MCconj}]
The proof is illustrated in Figure \ref{trimrectangle2},  for $m=5$, $n=7$, $h_1=4$ and $h_2=3$. Consider the rightmost subgraph $G_1$ of $TA_{m,n}^{h_1,h_2}$, which is  restricted by a dotted contour in Figure \ref{trimrectangle2}. By Graph Splitting Lemma \ref{GS}, we obtain
\begin{equation}
\M(TA_{m,n}^{h_1,h_2})=\M(TA_{m,n}^{h_1,h_2}-G_1)\M(G_1).
\end{equation}

Next, we consider the graph $G'$ obtained from $TA_{m,n}^{h_1,h_2}-G_1$ by removing horizontal forced edges (the circled ones on the right of $G_1$ in Figure \ref{trimrectangle2}). Applying the Graph-splitting Lemma \ref{GS} again to the second subgraph $G_2$ of $TA_{m,n}^{h_1,h_2}$, which is restricted by a  dotted contour, we have
\begin{equation}
\M(G')=\M(G'-G_2)\M(G_2).
\end{equation}
Repeat $i:=\lfloor\frac{h_1+1}{2}\rfloor-2$ more times the above process, we get a graph $\overline{G}$ and
\begin{equation}\label{trim1}
\M(TA_{m,n}^{h_1,h_2})=\M(\overline{G})\prod_{j=1}^{i}\M(G_j).
\end{equation}

Apply the same process for the lower part of $G$. We get a graph $\overline{\overline{G}}$ (the subgraph restricted by the bold contour in Figure \ref{trimrectangle2}) and
\begin{equation}\label{trim2}
\M(\overline{G})=\M(\overline{\overline{G}})\prod_{j=1}^{k}\M(H_j),
\end{equation}
where $k= \lfloor\frac{h_2+1}{2}\rfloor$, and $H_j$ is the $j$-th subgraph  (\textit{from the left}) restricted by a dotted contour in the lower part of $G$.

Combining (\ref{trim1}) and (\ref{trim2}), we deduce
\begin{equation}\label{trim3}
\M(TA^{(1)}_{m,n,h_1,h_2})=\M(\overline{\overline{G}})\prod_{j=1}^{i}\M(G_j)\prod_{j=1}^{k}\M(H_j).
\end{equation}

\begin{figure}\centering
\includegraphics[width=10cm]{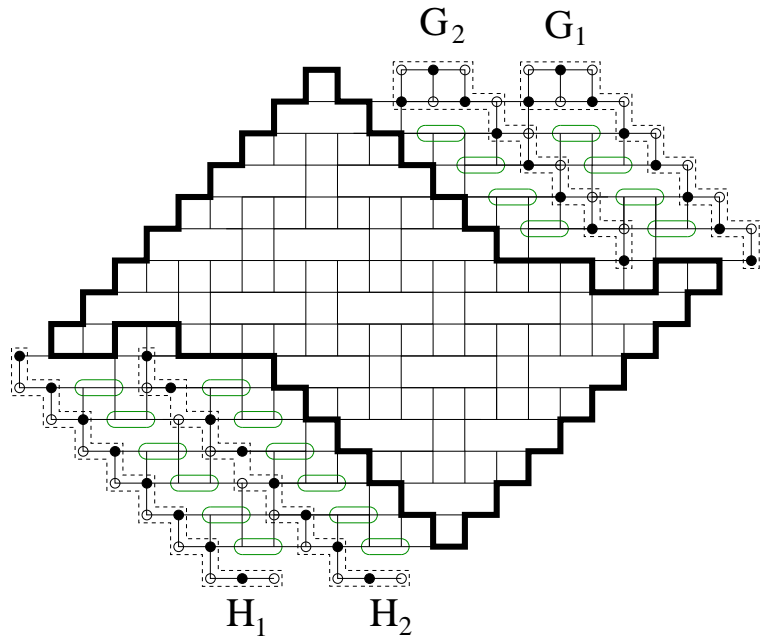}
\caption{Illustrating the proof of Theorem \ref{MCconj}.}
\label{trimrectangle2}
\end{figure}
It is easy to see that $M(G_j)=1$ if $h_1$ is odd, and $3$ if $h_1$ is even, for any $j=1,2,\dots, i$. Similarly, $M(H_j)=1$ if $h_2$ is odd, and $3$ if $h_2$ is even, for any $j=1,2,\dots,k$.
Thus,
\begin{equation}\label{trim4}
\M(TA_{m,n}^{h_1,h_2})=3^{\tau(h_1,h_2)}\M(\overline{\overline{G}}).
\end{equation}
On the other hand, $\overline{\overline{G}}$ is isomorphic to the graph $A^{(1)}_{a,b,c}$, where $a=n-\lfloor\frac{h_2+1}{2}\rfloor+1$,
$b=m-\lfloor\frac{h_2+1}{2}\rfloor+1$ and $c=\lfloor\frac{h_1+1}{2}\rfloor$. By (\ref{trim4}) and Theorem \ref{mainw1},  the equality (\ref{trimeq1}) follows.

The equality (\ref{trimeq2}) can be proved analogously.
%
%(b) This part is completely analogous, with the difference that  the value $\M(G_j)$ in (\ref{trim3}) is now $\frac{1}{2^{2b-a-2c}}$ if $h_1$ is odd, and  is $\frac{5}{2^{2b-a-2c}}$ if $h_1$ is even, for any $j=1,2,\dots,i$; and $\M(H_j)$ is $\frac{1}{2^{a}}$ if $h_2$ is odd, and  is $\frac{5}{2^{a}}$ if $h_2$ is even, for any $j=1,2,\dots,k$. Thus,
%\begin{equation}\label{trim5}
%\M(TA^{(1)}_{m,n,h_1,h_2})=5^{x}2^{-y}\M(\overline{\overline{G}}),
%\end{equation}
%for some non-negative integers $x$ and $y$. Explicitly, we have
% \begin{equation}
% x=
% \begin{cases}
% \frac{h_1+h_2}{2} &\text{if $h_1$ and $h_2$ are even;}\\
%\frac{h_1}{2}&\text{if $h_1$ is even and $h_2$ is odd;}\\
%\frac{h_2}{2}&\text{if $h_1$ is odd and $h_2$ is even;}\\
%0&\text{if $h_1$ and $h_2$ are odd,}
% \end{cases}
% \end{equation}
% and $y=(2b-a-2c)\lfloor\frac{h_1+1}{2}\rfloor+a\lfloor\frac{h_2+1}{2}\rfloor$.
%
%
%Moreover, with the new weight assignment, the graph $\overline{\overline{G}}$ is now isomorphic to the graph $A^{(1)}_{a,b,c}$ with the edges of shaded squares are weighted $1/2$. The latter graph is obtained from the dual graph of $D_{a,b,c}$ by applying Spider Lemma multiple times (see Figure \ref{TrandformMattBlum}). Thus,
%\begin{equation}\label{trim6}
%\M(D_{a,b,c})=2^{z}\M(\overline{\overline{G}})
%\end{equation}
%for some non-negative integer $z$. By enumerating explicitly, one gets $z=(c+d)(b-c)+c(c-1)/2+f(f-1)/2=(2b-a-c)(b-c)+c(c-1)/2+(2b-2a-c)(2b-2a-c-1)/2$. Therefore, equality (\ref{trimeq3}) follows from (\ref{trim5}), (\ref{trim6}) and Theorem \ref{oldmain}.
\end{proof}

\begin{figure}\centering
\includegraphics[width=12cm]{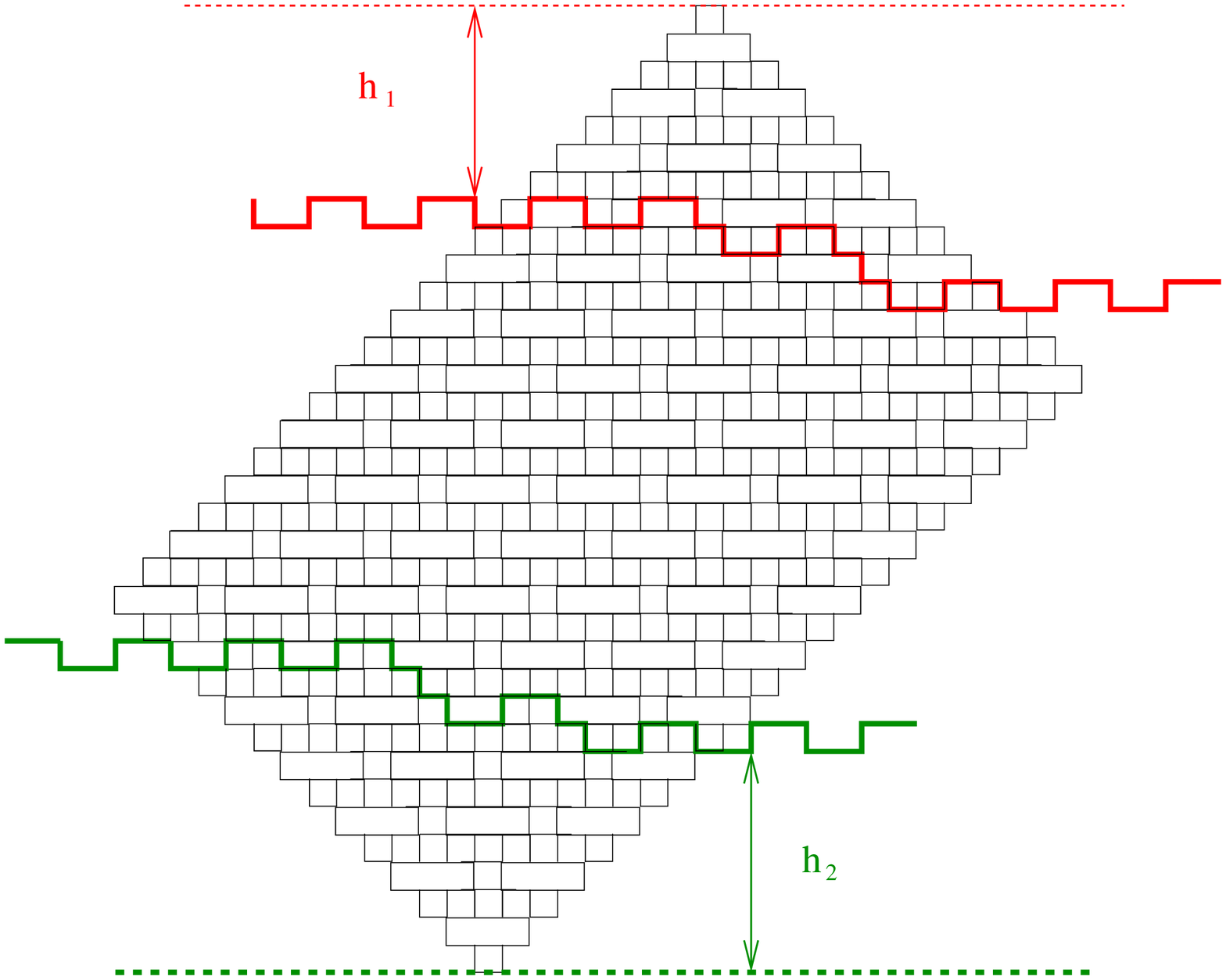}
\caption{}
\label{trimrectangle3}
\end{figure}

Next, we consider a variation of Theorem \ref{MCconj} as follows.  Instead of using horizontal trimming lines as in the Theorem \ref{MCconj}, we consider two new stair-shaped trimming lines. The structure of each level in the new trimming lines is similar to the old ones (i.e. is a zigzag line with alternatively bumps and holes of size $2$), and each two consecutive levels are connected by a ``staircase" (see Figure \ref{trimrectangle3}). Assume that $h_1$ is the distance between the top of the Aztec rectangle and the highest level of the upper trimming line, and $h_2$ is the distance between the bottom of the Aztec rectangle and the lowest level of the lower trimming line. Again, by the Graph Splitting Lemma \ref{GS}(a), we can cut off small subgraphs with the same structure as that of $G_i$ and $H_j$. We get again the final graph isomorphic to the graph $A^{(1)}_{a,b,c}$, with $a,b,c$ are defined as in the proof of the Theorem \ref{MCconj} (see Figure \ref{trimrectangle4}). Therefore, we have the following result.

\begin{thm}\label{trimgen}
 Assume that $\lfloor\frac{h_1+1}{2}\rfloor+\lfloor\frac{h_2+1}{2}\rfloor=2(n-m)$. Then
 the number of perfect matchings of the Aztec rectangle trimmed by two stair-shaped lines has all of its prime factors less than or equal to 13.
%
% (b) The statement in part (a) is still true when the edges of all squares between two consecutive $1\times 3$ brick are weighted $\frac{1}{2}$.
\end{thm}
The work of finding the explicit formulas for the numbers of perfect matchings (as well as the precise definition of the new trimmed Aztec rectangle) in Theorem \ref{trimgen} will be left as an exercise.

\begin{figure}\centering
\includegraphics[width=12cm]{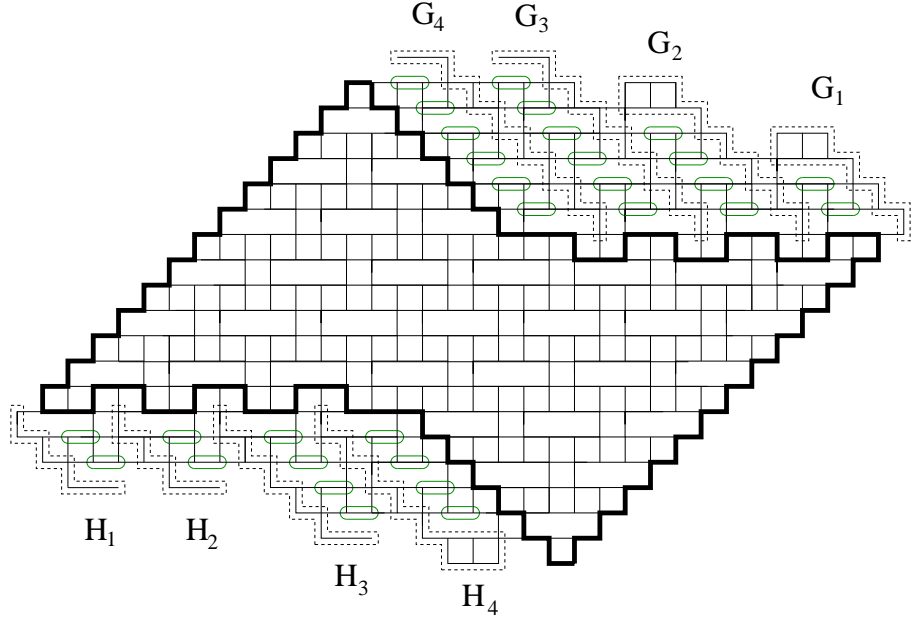}
\caption{Illustrating the proof of Theorem \ref{trimgen}.}
\label{trimrectangle4}
\end{figure}

\section{Relation between $TR_{a,b}$ and Hexagonal Dungeons}
%We observe that the number of perfect matchings of the trimmed augmented Aztec rectangle $TR_{a,b}$ in Theorem \ref{Dungeon} can be written by $10^{2a^2}11^{\lfloor\frac{a^2}{2}\rfloor}$, which is similar to the number of tilings ($13^{2a^2}14^{\lfloor\frac{a^2}{2}\rfloor}$) of the hexagonal dungeon $HD_{a,2a,b}$ introduced by Blum \cite{Propp} (see detail definition of the hexagonal dungeon in \cite{CL}). Figure \ref{hexagon} shows the hexagonal dungeon $HD_{2,4,6}$. This suggests us about a hidden relation between our graph and the hexagonal dungeon.
We consider the hexagonal dungeon $HD_{a,2a,b}$ introduced by Blum \cite{Propp} (see detail definition of the hexagonal dungeon in \cite{CL}). Figure \ref{hexagon} shows the hexagonal dungeon $HD_{2,4,6}$. We consider the dual graph of $HD_{a,2a,b}$ (i.e. the graph whose vertices are the small right triangles in $HD_{a,2a,b}$ and whose edges connect precisely two triangles sharing an edge), which is denoted by $G_{a,2a,b}$. The upper graph with solid edges in Figure \ref{hexagondw1} illustrate the dual graph of $HD_{2,4,6}$. It has been proven in \cite{CL} that the number of perfect matchings of $G_{a,2a,b}$ is given by $13^{2a^2}14^{\lfloor\frac{a^2}{2}\rfloor}$, which is similar to the number of perfect matchings of $TR_{a,b}$ ($10^{2a^2}11^{\lfloor\frac{a^2}{2}\rfloor}$). This suggests the existence of a hidden relation between $TR_{a,b}$ and the dual graph $G_{a,2a,b}$ of the hexagonal dungeon.

\begin{figure}\centering
\includegraphics[width=12cm]{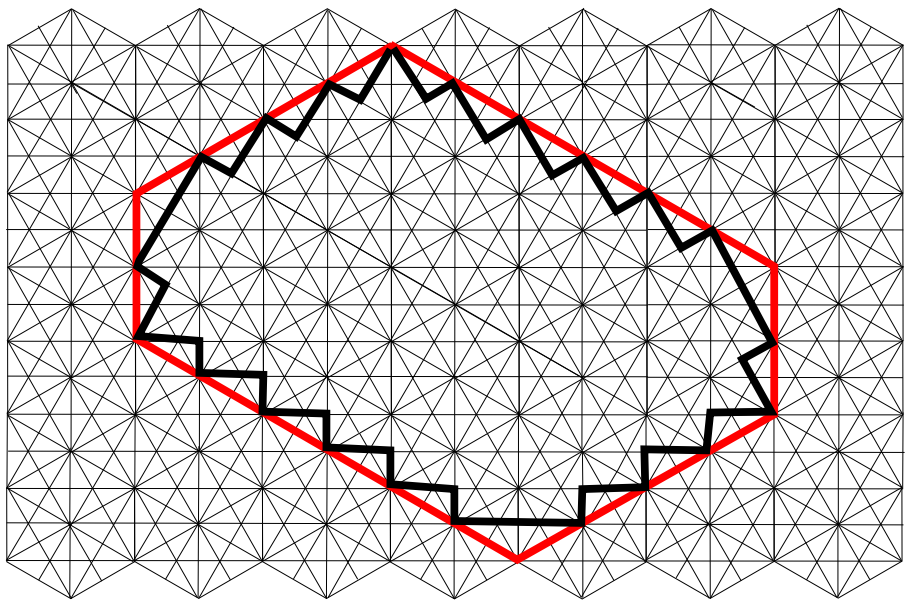}
\caption{The hexagonal dungeon of sides $2,$ $4,$ $6,$ $2,$ $4,$ $6$ (in cyclic order, starting from the western side). This Figure first appeared in
\cite{CL}.}
\label{hexagon}
\end{figure}

 If we assign some weights on the edges of a graph $G$, then we use the notation $\M(G)$ for the sum of weights of the perfect matchings of $G$, where the weight of a perfect matching is the product of weights of its constituent edges. We  call $\M(G)$ the \textit{matching generating function} of the weighted graph $G$.

 \begin{figure}\centering
\includegraphics[width=5cm]{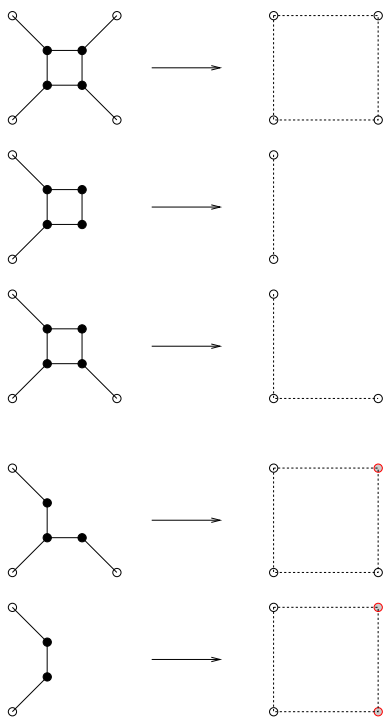}
\caption{Urban renewal trick.}
\label{urban}
\end{figure}

Next, we quote a well-known subgraph replacement trick called \textit{urban renewal}, which was first discovered by Kuperberg, and its variations found by Ciucu.

\begin{lem}[Urban renewal]\label{spider}
 Let $G$ be a weighted graph. Assume that $G$ has a subgraph $K$ as one of the graphs on the left column in Figure \ref{urban}, where only white vertices can have neighbors outside $K$, and where all edges have weight $1$. Let $G'$ be the weighted graph obtained from $G$ by replacing $K$ by its corresponding graph $K'$ on the right column of Figure \ref{urban}, where all dotted edges have weight $\frac{1}{2}$, and where the shaded vertices are the new ones which were not in $G$. Then we always have $\M(G)=2\M(G')$.
 \end{lem}

\begin{figure}\centering
\includegraphics[width=12cm]{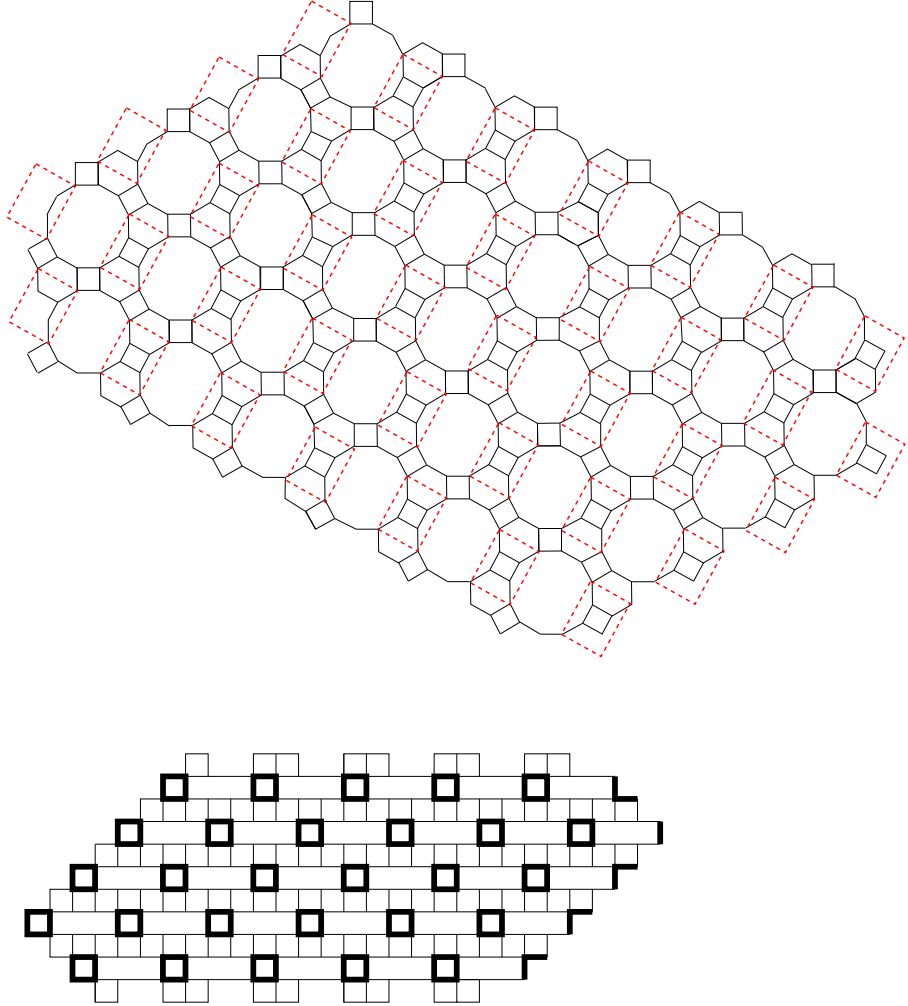}
\caption{Deforming the dual graph of $HD_{2,4,6}$ into a weighted graph on the square grid.}
\label{hexagondw1}
\end{figure}

Next, we apply suitable replacement rules in Lemma \ref{spider} to $G_{a,2a,b}$ around the dotted rectangles as in the upper graph in Figure \ref{hexagondw1}). Then we deform the resulting graph into a weighted graph $\overline{G}_{a,2a,b}$ on the square lattice (see the lower graph in Figure \ref{hexagondw1}; the bold edges have weight $\frac{1}{2}$).  We want to emphasize that even though the graphs $\overline{G}_{a,2a,b}$ and $TR_{a,b}$ have the same shape, their  weight assignments are \textit{different}. This means that Theorem \ref{Dungeon} can \textit{not} be deduced from the work in \cite{CL}.

We now want to consider the a common generalization of the weight assignments in the graphs $\overline{G}_{a,2a,b}$ and $TR_{a,b}$ as follows. Assume that $x,y,z$ are three indeterminate weights. We assign weights to edges of the grid $B$ so that each cross  pattern is weighted as in Figure \ref{weightassignment}.
Denote by $A^{(i)}_{a,b,c}(x,y,z)$'s and $F^{(i)}_{a,b,c}(x,y,z)$'s the corresponding weighted versions of the graphs $A^{(i)}_ {a,b,c}$'s and $F^{(i)}_{a,b,c}$'s, for $i=1,2,3$, respectively.

\begin{figure}\centering
\includegraphics[width=5cm]{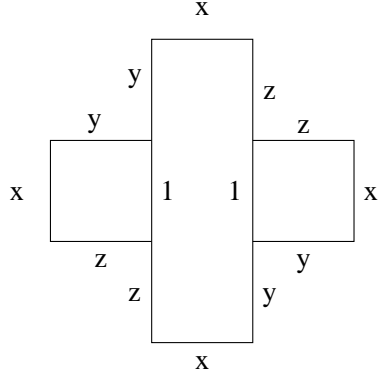}
\caption{The weight assignment for each cross.}
\label{weightassignment}
\end{figure}

Define the weighted versions of the functions $\alpha(a,b,c)$ and $\beta(a,b,c)$ by
\begin{equation}
\alpha(a,b,c; x,y,z)=
\begin{cases}
\frac{(x+2yz)yz}{x^{2}} &\text{if $3b+a-c \equiv 5 \pmod{6}$};\\
\frac{yz}{x} &\text{if $3b+a-c \equiv 3 \pmod{6}$};\\
\frac{x+yz}{x} &\text{if $3b+a-c \equiv 1 \pmod{6}$};\\
1&\text{otherwise}
\end{cases}
\end{equation}
and
\begin{equation}
\beta(a,b,c; x,y,z)=
\begin{cases}
\frac{(x+2yz)}{yz} &\text{if $3b+a-c \equiv 1 \pmod{6}$};\\
\frac{x}{yz} &\text{if $3b+a-c \equiv 3 \pmod{6}$};\\
\frac{(x+yz)x}{y^2z^2} &\text{if $3b+a-c \equiv 5 \pmod{6}$};\\
1&\text{otherwise.}
\end{cases}
\end{equation}
Our data suggests that
\begin{conj}\label{wconjecture}
The matching generating functions of the weighted
graphs\\ $A^{(i)}_{a,b,c}(x,y,z)$'s all have form
\begin{equation}\label{geneq1}
\alpha(a,b,c; x,y,z) 2^{X}(x^2+2xyz+2y^2z^2)^{Y} (2x^2+5xyz+4y^2z^2)^{Z}x^{T} y^{Q}z^{K},
\end{equation}
for some $X,Y,Z,T,Q,K$ depending on only $a,b,c$. Similarly, the matching generating functions of $F^{(i)}_{a,b,c}(x,y,z)$'s all have form
\begin{equation}\label{geneq2}
\beta(a,b,c;x,y,z) 2^{X'}(x^2+2xyz+2y^2z^2)^{Y'} (2x^2+5xyz+4y^2z^2)^{Z'}x^{T'} y^{Q'}z^{K'},
\end{equation}
for some $X',Y',Z',T',Q',K'$ depending on only $a,b,c$.
\end{conj}

Denote by $TR_{a,b}(x,y,z)$ the corresponding weighted version of  the graph $TR_{a,b}$.
If  Conjecture \ref{wconjecture} is true, then by the graph-splitting trick in the proof of Theorem \ref{Dungeon}, we can implies that the matching generating function of $TR_{a,b}(x,y,z)$ is also given by powers of $2$, $x$, $y$, $z$, $(x^2+2xyz+2y^2z^2)$, and  $(2x^2+5xyz+4y^2z^2)$.


\begin{thebibliography}{12}

\bibitem{Ciucu}
M. Ciucu,
\emph{Perfect matchings and perfect powers},
J. Algebraic Combin. \textbf{17} (2003), 335--375.

\bibitem{CL}
M. Ciucu and T. Lai,
\emph{Proof of Blum's Conjecture on Hexagonal Dungeons},
 J. Combin. Theory Ser. A \textbf{125}  (2014), 273--305.



\bibitem{Elkies}
N. Elkies, G. Kuperberg, M. Larsen, and J. Propp,
\emph{Alternating-sign matrices and domino tilings (Part I)}, J. Algebraic Combin. \textbf{1} (1992), 111--132.


\bibitem{Krat}
C. Krattenthaler. \newblock  Schur function identities and the number of perfect matchings of holey Aztec rectangles. \newblock
 \textit{q-series from a contemporary perspective}: 335--349, 1998. \newblock \textit{Contemp. Math.} \textbf{254}, \textit{Amer. Math. Soc.}, Providence, RI, 2000.

\bibitem{Kuo}
E. H. Kuo,
\emph{Applications of Graphical Condensation for Enumerating Matchings and Tilings},
Theor. Comput. Sci. \textbf{319} (2004),
29--57.


\bibitem{Mac}
P. A. MacMahon,
\emph{Combinatory Analysis}, vol. 2,
Cambridge Univ. Press, 1916, reprinted by Chelsea, New York, 1960.

\bibitem{Mui}
T. Muir, \textit{The Theory of Determinants in the Historical Order of Development}, vol. I, Macmil-
lan, London, 1906.

\bibitem{Tri1} T. Lai, \textit{Enumeration of hybrid domino-lozenge tilings}, J. Combin. Theory Ser. A \textbf{122}, 2014, 53--81.

\bibitem{Tri2} T. Lai, \emph{New aspects of regions whose tilings are enumerated by perfect powers}, Elec. J. of Combin. \textbf{20} (4) (2013), P31.




\bibitem{Propp}
J. Propp,
\emph{Enumeration of matchings: Problems and progress},
New Perspectives in Geometric Combinatorics,  Cambridge Univ. Press, 1999, 255--291.


\bibitem{SZ}
H. Sachs and H. Zernitz, \textit{Remark on the dimer problem}, Discrete Appl. Math. \textbf{51} (1994), 171--179.

\bibitem{Stanley}
R. Stanley, \textit{Enumerative combinatorics},  Vol 2, Cambridge Univ. Press 1999.


\bibitem{Yang}
B.-Y. Yang,
\emph{Two enumeration problems about Aztec diamonds},
Ph.D. thesis, Department of Mathematics, Massachusetts Institute of Technology, MA, 1991.

\end{thebibliography}
\end{document}